\documentclass[12pt]{amsart}

\usepackage[margin=1in]{geometry}
\usepackage[all]{xy}
\usepackage{enumitem}
\usepackage{mathrsfs}
\usepackage{amssymb}
\usepackage[colorlinks=true,linkcolor=blue,urlcolor=blue]{hyperref}

\makeatletter
\newlength{\negph@wd}
\DeclareRobustCommand{\negphantom}[1]{%
  \ifmmode
    \mathpalette\negph@math{#1}%
  \else
    \negph@do{#1}%
  \fi
}
\newcommand{\negph@math}[2]{\negph@do{$\m@th#1#2$}}
\newcommand{\negph@do}[1]{%
  \settowidth{\negph@wd}{#1}%
  \hspace*{\negph@wd}
}
\makeatother

\def\logGm{\mathbf{G}_{\log}}
\def\Gm{\mathbf{G}_m}
\def\Hom{\operatorname{Hom}}
\def\bExt{\operatorname{\mathbf{Ext}}}
\def\bExtPan{\operatorname{\mathbf{ExtPan}}}
\def\Ext{\operatorname{Ext}}
\def\gr{\operatorname{gr}}
\def\fppf{\mathrm{fppf}}
\def\logfppf{\mathrm{\log\!.fppf}}
\def\et{\mathrm{\acute{e}t}}
\def\loget{\mathrm{\log\!.\acute{e}t}}
\def\KN{\mathrm{KN}}
\def\Spec{\operatorname{Spec}}
\def\uExt{\underline{\Ext}}
\def\uHom{\underline{\Hom}}

\def\tropGm{\mathbf G_{\rm trop}}
\def\Pic{\operatorname{Pic}}
\def\uBilin{\undernorm{\operatorname{Bilin}}}
\def\uBiext{\undernorm{\operatorname{Biext}}}
\def\image{\operatorname{image}}

\mathchardef\mhyphen="2D

\newtheorem{theorem}{Theorem}
\newtheorem{proposition}[theorem]{Proposition}
\newtheorem{corollary}[theorem]{Corollary}
\newtheorem{lemma}[theorem]{Lemma}

\theoremstyle{remark}
\newtheorem{remark}[theorem]{Remark}

\theoremstyle{definition}

\newtheorem{example}[theorem]{Example}

\numberwithin{theorem}{section}
\numberwithin{equation}{section}

\def\overnorm#1{\overline{#1}\vphantom{#1}}
\def\undernorm#1{\underline{#1}\vphantom{#1}}

\title{The monodromy pairing for logarithmic $1$-motifs}

\author{Jonathan Wise}
\email{jonathan.wise@colorado.edu}
\address{University of Colorado, Campus Box 395, Boulder, CO 80309-0395}

\begin{document}

\maketitle

\begin{abstract}
We describe a $3$-step filtration on all logarithmic abelian varieties with constant degeneration.  The obstruction to descending this filtration, as a variegated extension, from logarithmic geometry to algebraic geometry is encoded in a bilinear pairing valued in the characteristic monoid of the base.  This pairing is realized as the monodromy pairing in $p$-adic, $\ell$-adic, and Betti cohomologies, and recovers the Picard--Lefschetz transformation in the case of Jacobians.  The Hodge realization of the filtration is the monodromy weight filtration on the limit mixed Hodge structure.
\end{abstract}

\section{Introduction}
\label{sec:intro}

Let $R$ be a strictly henselian discrete valuation ring with fraction field $K$ and suppose that $A_K$ is an abelian variety over $K$ that extends to a group scheme $A$ over $R$ with semistable reduction.  Let $X$ be the character lattice of the toric part $A^{\rm tor}$ of the closed fiber, and let $Y$ be the character lattice of the toric part of the dual semistable degeneration.  For each prime number $\ell$, Grothendieck defined a bilinear pairing \eqref{eqn:33}, with the subscript $\ell$ denoting completion at $\ell$~\cite[\S9]{sga7-IX}:
\begin{equation} \label{eqn:33}
Y_\ell \times X_\ell \to \mathbf Z_\ell
\end{equation}
When $\ell$ is invertible in $R$, this pairing can be defined in terms of the monodromy action of the inertia group on the $\ell$-adic Tate module of $A$.  

Grothendieck gives a second construction, valid for all $\ell$, that we summarize below.  The Tate module $U$ has a $3$-step filtration, $U_0 \subset U_1 \subset U_2 = U$, where $U_0$ is the Tate module of the toric part of the special fiber, and $U_1$ is the Tate module of the special fiber, also the part fixed by the action of monodromy.  The quotient $U_1/U_0$ is the Tate module of the abelian part and $U_2/U_1$ is the moving part.  The filtration may be variegated into two extensions,
\begin{gather*}
0 \to U_0 \to U_1 \to U_1/U_0 \to 0 \\
0 \to U_1/U_0 \to U_2/U_0 \to U_2/U_1 \to 0 
\end{gather*}
both of which extend over all of $R$.  The monodromy pairing arises as the obstruction to extending $U$ over $R$ as a \emph{variegated extension} (see Section~\ref{sec:variegated}).

It is not clear from Grothendieck's definition that the monodromy pairing should be integrally defined, nor that it should be independent of $\ell$.  Grothendieck proves both of these assertions~\cite[Th\'eor\`eme~10.4]{sga7-IX}, by a combination of reduction to Jacobians of curves and reduction to characteristic zero, where he may deduce an explicit tropical formula for the monodromy pairing in terms of the Picard--Lefschetz transformation of a double point degeneration.  

In this paper we will reconstruct the monodromy pairing using logarithmic and tropical considerations and then use it to recover the formula for the Picard--Lefschetz transformation.  A semistable degeneration of abelian varieties naturally produces a \emph{logarithmic abelian variety} on the special fiber.  This logarithmic abelian variety has the structure of a quotient of a logarithmic semiabelian variety by a lattice, and this will be our starting point in this paper.  For a construction of this uniformization in the language of logarithmic geometry, we direct the reader to \cite[\S\S13--15]{KKN4} or \cite[Theorem~1.1]{Zhao}, which reduce the logarithmic construction to the algebraic one \cite{Mumford-72b, Raynaud-ICM, BL1, BL2, FC}.  For Jacobians of logarithmic curves, an explicit formula for the uniformization appears in \cite{logpic}, and we will give yet another construction for a degenerating abelian variety in forthcoming work.

Let $A = A_1 / Y$ be a logarithmic abelian variety, the quotient of a logarithmic semiabelian 
variety $A_1$ by a lattice $Y$.  Let $A_0$ be the (logarithmic) toric part of $A_1$ and let $A_2 = A$.  Then 
the sequence of homomorphisms
\begin{equation*}
A_0 \to A_1 \to A_2
\end{equation*}
can be viewed as a $3$-step filtration of $A$, in the sense of strictly commutative $2$-groups.  This
filtration can be variegated into two exact sequences
\begin{gather*}
0 \to A_0 \to A_1 \to A_1 / A_0 \to 0 \\
0 \to A_1/A_0 \to A_2/A_0 \to A_2/A_1 \to 0
\end{gather*}
giving $A$ the structure of a variegated extension.  Both of these sequences are induced from
exact sequences of sheaves on the flat site of $S$, and we therefore obtain an obstruction to 
descent of $A$, as a variegated extension, to the flat site of $S$.  
In Section~\ref{sec:monodromy}, we recognize this obstruction as a bilinear pairing between $X$ and $Y$, 
valued in $\overnorm M_S^{\rm gp}$, where $X$ is the character lattice of $A_0$ and $Y$ is the kernel of 
$A_1 \to A$.

The bulk of this paper is devoted to constructing the $p$-adic, $\ell$-adic, Betti, and Hodge 
realizations of this obstruction, in Sections~\ref{sec:torsion}, \ref{sec:etale}, \ref{sec:betti}, and \ref{sec:hodge}, respectively.  In all of those cases, 
the filtration on $A$ becomes an honest filtration (i.e., of abelian groups and not only of 
strictly commutative $2$-groups).  In the $p$-adic and $\ell$-adic contexts, we obtain Grothendieck's 
construction.  The Hodge realization recovers the limit mixed Hodge structure on $H^1$.  In the 
case of a Jacobian, the monodromy pairing encodes the Picard--Lefschetz transformation on 
$H^1$.  We explain in Section~\ref{sec:picard-lefschetz} how the \'etale and Betti realizations of the monodromy pairing 
recover the Picard--Lefschetz transformations in those settings.

The calculations in Section~\ref{sec:torsion} prove the integrality of Grothendieck's monodromy pairing 
directly, without the need for reduction to characteristic zero.  (The reduction to 
Jacobians is hidden in the construction of the logarithmic uniformization, which relies on a 
reduction to Jacobians in at least some constructions \cite{BL2}; we plan to discuss this in greater detail elsewhere.)  Combined with the construction of the logarithmic Jacobian in~\cite{logpic}, 
our calculations also recover the explicit formula for the monodromy pairing of a Jacobian (Corollary~\ref{cor:curves}), again without the need to restrict to characteristic zero.

The integrality of the monodromy pairing can also be demonstrated very naturally using non-archimedean analytification~\cite{Raynaud}.  Baker and Rabinoff extend this construction to valued bases of rank~$1$ that are not necessarily discretely valued (see also~\cite{L}), as well as give a tropical interpretation of the monodromy pairing~\cite{BR}.  
The construction presented here is, of course, closely related to those but is more general in the following senses:  the logarithmic structures on the base are fine and saturated, but are otherwise arbitrary, which in particular permits bases of rank greater than~$1$ that need not necessarily be valuative; the construction of the logarithmic monodromy pairing depends only on the central fiber of a degeneration and its logarithmic structure (these  data might reasonably be called the logarithmic general fiber), which are determined by, but contain strictly less information than, the analytic general fiber of the degeneration.

Over a $1$-parameter base, closely related constructions have appeared in works of Gillibert \cite{Gillibert} and of Zhao \cite{Zhao}.  Gillibert works with the N\'eron models of dual abelian varieties over a discrete valuation ring, and constructs a canonical limit of the biextension of the general fibers; Zhao does the same, except he works with tamely ramified $1$-motifs with potentially good reduction and finds extensions to $1$-motifs in the logarithmic \'etale site.  Zhao's definition of the monodromy pairing \cite[Section 4]{Zhao} is essentially the same as our construction in Section \ref{sec:log-1-motifs} (apart from the restriction to a $1$-parameter base, which does not materially affect the definition).  Both Gillibert and Zhao view the monodromy pairing as an obstruction to descent of a canonical biextension (by $\Gm$) in the Kummer logarithmic flat or Kummer logarithmic \'etale topology to the fppf or \'etale topology.%
\footnote{The monodromy pairing considered by Gillibert is actually a closely related pairing (also defined by Grothendieck~\cite[Th\'eor\`eme 7.2 b]{sga7-VIII}) between the component groups of the N\'eron models of dual abelian varieties, and valued in $\mathbf Q / \mathbf Z$.  We will explain the relationship between this pairing and the monodromy pairing as we have defined it in Section~\ref{sec:dual}.}
The perspective adopted here is that the monodromy pairing is the obstruction to descending a logarithmic variegated extension to an algebraic one,
but we will explain the relationship to Gillibert's and Zhao's approaches in Section~\ref{sec:dual}.%
\footnote{Grothendieck's approach to the monodromy pairing through variegated extensions does not require explicit invocation of the dual abelian variety.  Thus Section~\ref{sec:dual} and appendix~\ref{sec:log-trop-Gm} are only used for comparison to the approaches of \cite{Gillibert} and \cite{Zhao}, and the remainder of the paper is independent of them.}
Zhao also constructs the \'etale realization of a Kummer \'etale logarithmic $1$-motif, which closely parallels our Section~\ref{sec:etale}.

\subsection*{Conventions}

A \emph{logarithmic scheme} will always be a scheme with a fine and saturated logarithmic structure.  In Section~\ref{sec:betti}, we will make use of a scheme with a non-integral logarithmic structure to construct the Kato--Nakayama space, but we will call this a `scheme with logarithmic structure' rather than a `logarithmic scheme'.  

The \emph{logarithmic multiplicative group}, $\logGm$, is the functor on logarithmic schemes given by $\logGm(S) = \Gamma(S, M_S^{\rm gp})$.  The \emph{tropical multiplicative group} $\tropGm$ is the quotient $\logGm / \Gm$ (the quotient is taken in the strict \'etale topology on logarithmic schemes), with $\tropGm(S) = \Gamma(S, \overnorm M_S^{\rm gp})$, where $\overnorm M_S = M_S / \mathcal O_S^\ast$ is the characteristic monoid and $\overnorm M_S^{\rm gp}$ is its associated group.  The operation of $\logGm$ is written multiplicatively, while the operation of $\tropGm$ is written additively.  See Appendix~\ref{sec:log-trop-Gm} for some basic properties of $\logGm$ and $\tropGm$.

We typically use a bold symbol to denote a category or stack and the same symbol in roman for its sheaf of isomorphism classes.  For example, $\bExt^1(A,B)$ is the groupoid of extensions of $A$ by $B$, while $\Ext^1(A,B)$ is the set of isomorphism classes of extensions.

We must consider sheaves on many different sites, including small and large sites, as well as logarithmic versions.  These are connected by canonical morphisms with fully faithful pullback functors, so it is harmless, and often convenient, to identify sheaves on smaller sites with their pullbacks to larger sites.  For example, a sheaf on the small \'etale site of a scheme may be pulled back to a sheaf on the large \'etale site; this pullback is an algebraic space, and we refer to it as the original sheaf's espace \'etal\'e.  Likewise, we frequently pull back sheaves from the big \'etale site of the underlying scheme of a logarithmic scheme $S$ to the site of logarithmic schemes over $S$ with strict \'etale covers; when pulling back a sheaf that is representable by a scheme, this corresponds to equipping that scheme with the logarithmic structure pulled back from $S$.  I have tried to point out these identifications by pullback when they occur.  Without qualitification, $\uHom(F,G)$ and $\uExt^1(F,G)$ will denote the sheaves of homomorphisms from $F$ to $G$ and extensions of $F$ by $G$, respectively; we will affix a subscript to denote the restriction of these sheaves to smaller sites (e.g., $\uHom_{S^{\rm KN}}(-,-)$ and $\uExt^1_{S^{\rm KN}}(-,-)$ in Section~\ref{sec:hodge}).

We work frequently with stacks of strictly commutative $2$-groups, which are also known as strictly commutative Picard stacks \cite[1.4]{sga4-XVIII}.  Not all of the operations we perform on these objects are described in op.\ cit., but further details may be found elsewhere in the literature \cite{Ber11, BB19, Bro21, 2ab}.  Moreover, by the equivalence between strictly commutative $2$-groups and $2$-term complexes \cite[Lemma~1.4.13]{sga4-XVIII}, all of these operations have analogues in the equivalent $2$-category of $2$-term complexes.

For example, the \emph{cokernel}, or \emph{quotient}, or a morphism of strictly commutative $2$-groups $A \to B$ is a special case of the fibered sum construction of \cite[Definition~2.2]{Ber11}: in the notation of op.\ cit., the quotient $[B/A]$ is $B +^A 0$.  In the special case where $A$ and $B$ are sheaves of groups, the underlying stack of $[B/A]$ is simply the stack quotient of $B$ by $A$; its group structure makes it the universal strictly commutative $2$-group receiving a morphism from $B$ and equipped with a specified isomorphism between the composition $A \to B \to [B/A]$ and the zero morphism.  The quotient may also be obtained by representing strictly commutative $2$-groups as complexes and then truncating the cone of the corresponding morphism.  If this is done naively, it raises questions of well-definition, because the cone is not functorial; however, when the equivalence between $2$-term complexes and strictly commutative $2$-groups is formulated carefully, the quotient may be obtained by an explicit construction \cite[\S8]{2ab}.

\section{Logarithmic $1$-motifs}
\label{sec:log-1-motifs}

Let $S$ be a logarithmic scheme.  By a \emph{logarithmic torus} over $S$, we will mean a group-valued functor on logarithmic schemes over $S$ of the form $\uHom(X, \logGm)$, where $X$ is an \'etale sheaf of abelian groups on $S$ that is locally finitely generated and free (henceforth, a \emph{lattice}).  We also regard lattices as sheaves on the strict \'etale site of logarithmic schemes over $S$ by pullback from the small \'etale site; this pullback is representable by a logarithmic scheme with a strict projection to $S$, and we denote this logarithmic scheme by the same symbol.  A \emph{logarithmic semiabelian variety} over $S$ is an extension of an abelian variety over $S$ (with strict logarithmic structure, relative to $S$) by a logarithmic torus over $S$.  In this paper, a \emph{logarithmic $1$-motif} will be the strictly commutative $2$-group quotient of a logarithmic semiabelian variety by a sheaf of lattices.

\begin{remark}
This definition is not the correct one if one wants to allow logarithmic $1$-motifs to vary in families with non-constant degeneracy.  In order to be able to algebraize formal deformations, we need a condition like~\cite[3.1]{KKN2}.  However, every logarithmic $1$-motif (with constant degeneracy) in the sense of~\cite{KKN2} induces one in the sense defined here, so we omit further discussion of~\cite[3.1]{KKN2}.
\end{remark}

If $G$ is a semiabelian variety with torus part $T = \uHom(X,\Gm)$ then the inclusion $\Gm \to \logGm$ induces $T \to T^{\log} = \uHom(X,\logGm)$, and therefore $G \to G^{\log}$ by pushout.  This determines a functor from the category of semiabelian varieties over $S$ to the category of logarithmic semiabelian varieties over~$S$.

Lemmas~\ref{lem:discrete-sect-tors} and~\ref{lem:discrete-hom-ext}, and Proposition~\ref{prop:category}, below, follow the proof of \cite[Lemma~2.3.2]{Raynaud}, which is attributed by Raynaud to Illusie.  The lemma also appears in a special case as \cite[Proposition~7.22]{KKN2}.

\begin{lemma} \label{lem:discrete-sect-tors}
	Let $G$ be a smooth scheme with connected geometric fibers over $S$ and let $Z$ be the espace \'etal\'e of a sheaf of torsion-free abelian groups on the small \'etale site of $S$.  Then all morphisms $G \to Z$ and all $Z$-torsors over $G$ descend uniquely to $S$.
\end{lemma}
\begin{proof}
	Suppose that $\varphi : G \to Z$ is a morphism.  Let $g$ be a geometric point of $G$ lying above a geometric point $s$ of $S$, and let $z = \varphi(g)$.  Since $Z$ is \'etale over $S$, there is an \'etale neighborhood $U$ of $s$ in $S$ and a section of $Z$ over $U$, the image of whose restriction to $s$ is $z$.  Let $Z_0 \subset Z$ be the image of this section.  Then $Z_0$ is open in $Z$, so $G_0 = \varphi^{-1} Z_0 \subset G$ is an open subset whose restriction to every fiber of $G$ over $U$ is also closed in the fiber.  But $G$ has connected fibers over $S$, so $G_0 = G$.  Thus the restriction of $\varphi$ to the preimage of $U$ in $G$ factors uniquely through $U$.  Repeating this for every geometric point $g$ of $G$, we conclude that $\varphi$ factors through $S$.

	Now suppose that $E$ is a $Z$-torsor over $G$.  Then $E$ is an algebraic space and is smooth over $S$.  By \cite[Theorem~4.5]{minimal}, there is an \'etale algebraic space $\pi_0(E/S)$ over $S$ and an $S$-morphism $E \to \pi_0(E/S)$, such that the fiber of $\pi_0(E/S)$ over a geometric point $s$ of $S$ is $\pi_0(E_s)$.  We will show that $E \to \pi_0(E/S) \mathop\times_S G$ is an isomorphism.

	Since $Z$ and $\pi_0(E/S)$ are both \'etale sheaves over $S$, it is sufficient to check this on the geometric fibers over $S$.  We are therefore reduced to the situation where $S$ is the spectrum of an algebraically closed field, in which case $G$ is regular and $Z$ is a discrete abelian group.  The underlying $Z$-torsor of $E$ is therefore trivial \cite[Proposition~5.1]{sga7-VIII}, and $Z \to \pi_0(G \times Z)$ is therefore an isomorphism.
\end{proof}

\begin{lemma} \label{lem:discrete-hom-ext}
	Let $G$ be a smooth commutative group scheme with connected geometric fibers over $S$ and let $Z$ be the espace \'etal\'e of a sheaf of torsion-free abelian groups on the small \'etale site of $S$.  Then the homomorphisms of $G$ into $Z$ and extensions of $G$ by $Z$ are all trivial:
	\setcounter{equation}{\value{theorem}}
	\begin{equation}
		\uHom(G,Z) = \uExt^1(G,Z) = 0
	\end{equation}
\end{lemma}
\begin{proof}
	By Lemma~\ref{lem:discrete-sect-tors}, any homomorphism $G \to Z$ factors uniquely through $S$, hence must be zero.  Also by Lemma~\ref{lem:discrete-sect-tors}, the underlying $Z$-torsor of an extension $E$ of $G$ by $Z$ must be trivial.  Therefore the structure of an extension is encoded by a map $G \times G \to Z$ describing the group structure of $E$.  Again by Lemma~\ref{lem:discrete-sect-tors}, this must be constant, so $E$ is a trivial extension.
\end{proof}

\setcounter{theorem}{\value{equation}}
\begin{proposition} \label{prop:semiabelian}
	Every logarithmic semiabelian variety is induced from a semiabelian variety 
	that is unique up to unique isomorphism.
\end{proposition}
\begin{proof}
	Let $A$ be an abelian variety over $S$ and let $X$ be a lattice over $S$.  We use the following exact sequence:
	\begin{multline*}
		\Hom(A, \uHom(X, \tropGm)) \to \Ext^1(A, \uHom(X, \Gm)) \\ 
		\to \Ext^1(A, \uHom(X, \logGm)) \to \Ext^1(A, \uHom(X, \tropGm))
	\end{multline*}
	We wish to show that the outer terms vanish.  
	Over any logarithmic $S$-scheme $T$, a morphism $A \to \uHom(X,\tropGm)$ 
	corresponds to a section of the espace \'etal\'e of 
	$\uHom(X,\overnorm M_T^{\rm gp})$ over $A$, from which it is determined.
	But this is is necessarily zero by Lemma~\ref{lem:discrete-hom-ext}.  Likewise, 
	an extension restricts to an extension of $A$ by $\uHom(X,\overnorm M_T^{\rm gp})$, 
	which is uniquely split by Lemma~\ref{lem:discrete-hom-ext}.  Since the splitting 
	is unique, it commutes with base change, which implies that 
	$\Ext^1(A, \uHom(X,\tropGm)) = 0$ as well.  
\end{proof}

Let $G$ be a logarithmic semiabelian variety and let $G^{\rm alg}$ be its
underlying semiabelian variety.  We define $G^{\rm trop}$ to be the quotient
$G/G^{\rm alg}$.  If the toric part of $G$ is $\uHom(X, \logGm)$, then 
$G^{\rm trop} = \uHom(X, \tropGm)$.

Suppose that $G$ is a logarithmic semiabelian variety whose toric part has
character lattice~$X$.  Let $Y$ be a lattice over $S$, and let $Y \to G$
be a homomorphism.  This induces a homomorphism $Y \to G^{\rm trop}$ and
therefore a pairing $X \otimes Y \to \tropGm$, to be denoted 
$\langle x,y \rangle$ on local sections $x$ of $X$ and $y$ of $Y$.  This
pairing will turn out to be the \emph{logarithmic monodromy pairing}.

\begin{proposition} \label{prop:category}
	For logarithmic $1$-motifs $[G/Y]$ and $[G'/Y']$, the morphisms 
	$[G/Y] \to [G'/Y']$ are the same as the morphisms of complexes 
	$[Y \to G] \to [Y' \to G']$.  
\end{proposition}
\begin{proof}
	Suppose that $[G/Y] \to [G'/Y']$ is a homomorphism.  The choices of
	lift of the composed homomorphism $G \to [G'/Y']$ to $G \to G'$ form a
	pseudotorsor on $S$ under $\uHom(G,Y')$.  By
	Lemma~\ref{lem:discrete-hom-ext}, we have $\uHom(G,Y') = 0$, so lifts
	are unique if they exist, and the existence of a lift is a local
	question in $S$.

	The obstruction to finding a lift is the extension of $G$ by $Y'$
	obtained by pullback from $G \to [G'/Y']$.  By
	Lemma~\ref{lem:discrete-hom-ext}, all such extensions are uniquely
	split.
\end{proof}

The proposition implies that logarithmic $1$-motifs form a category, and not
merely a $2$-category.

Let $M = [Y \to G]$ be a logarithmic $1$-motif, where $G$ is an extension of
the abelian variety $A$ by the logarithmic semiabelian variety 
$T = \uHom(X, \logGm)$.  We give $M$ the following filtration, called the
\emph{weight filtration}:
\setcounter{equation}{\value{theorem}}
\begin{align}
W_{-1} M & = T \notag \\
W_{0} M & = G \label{eqn:3} \\
W_1 M & = M  \notag
\end{align}
The following are the graded pieces of this filtration:
\begin{align*}
\gr_{-1} W_\bullet M & = T \\
\gr_0 W_\bullet M & = A \\
\gr_1 W_\bullet M & = \mathrm B Y
\end{align*}
Here $\mathrm BY$ is the classifying stack of $Y$-torsors, with its $2$-group 
structure coming from contracted sum of torsors.  It is representable by the 
complex $[Y \to 0 ]$.

Suppose $G$ is a logarithmic semiabelian variety with toric part $T$, induced from
a semiabelian variety $G^{\rm alg}$ with toric part $T^{\rm alg}$, and that
$G^{\rm trop} = G/ G^{\rm alg} = T/T^{\rm alg}$ is its tropicalization.  If
$M = [G/Y]$ is a logarithmic $1$-motif then we write $M^{\rm trop} = [G^{\rm trop}/Y]$
and call it the \emph{tropicalization} of $M$.  The filtration~\eqref{eqn:3}
induces a filtration of $M^{\rm trop}$:
\begin{align*}
	W_{-1} M^{\rm trop} & = T^{\rm trop} \\
	W_0 M^{\rm trop} & = T^{\rm trop} \\
	W_1 M^{\rm trop} & = M^{\rm trop}
\end{align*}
Here are the associated graded pieces of $W_\bullet M^{\rm trop}$:
\begin{align*}
	\gr_{-1} W_\bullet M^{\rm trop} & = T^{\rm trop} \\
	\gr_0 W_\bullet M^{\rm trop} & = 0 \\
	\gr_1 W_\bullet M^{\rm trop} & = \mathrm B Y
\end{align*}

\section{Variegated extensions}
\label{sec:variegated}

We recall the theory of variegated extensions (extensions panach\'ees) from \cite[9.3]{sga7-IX}, with a few complements.  Fix sheaves of abelian groups $P$, $Q$, and $R$ and extensions~\eqref{eqn:100} and~\eqref{eqn:101}:
\begin{gather}
	0 \to P \to F \to R \to 0 \label{eqn:100} \tag{$F$} \\
	0 \to R \to E \to Q \to 0 \label{eqn:101} \tag{$E$}
\end{gather}
A \emph{variegated extension} of $E$ by $F$ is a sheaf of abelian groups $W$ with a $3$-step filtration
\begin{equation}
W = W_1 \supset W_0 \supset W_{-1} \supset 0
\end{equation}
and identifications of the extensions~\eqref{eqn:100} and~\eqref{eqn:101}, respectively, with~\eqref{eqn:58} and~\eqref{eqn:59}:
\begin{gather}
0 \to W_{-1} \to W_0 \to W_0/W_{-1} \to 0 \label{eqn:58} \\
0 \to W_0/W_{-1} \to W_1/W_{-1} \to W_1/W_0 \to 0 \label{eqn:59}
\end{gather}
In other words, a variegated extension of $E$ by $F$ is a commutative diagram~\eqref{eqn:62}, with exact rows and columns, and $P = W_{-1}$, $F = W_0$, $W = W_1$:
\begin{equation} \label{eqn:62} \vcenter{ \xymatrix{
& 0 \ar[d] & 0 \ar[d] \\
0 \ar[r] & P \ar[r] \ar[d] & P \ar[r] \ar[d] & 0  \ar[d] \\
0 \ar[r] & F \ar[r] \ar[d] & W \ar[r] \ar[d] & Q \ar[r] \ar[d] & 0 \\
0 \ar[r] & R \ar[r] \ar[d] & E \ar[r] \ar[d] & Q \ar[r] \ar[d] & 0 \\
& 0 & 0 & 0
}} \end{equation}

A variegated extension of strictly commutative $2$-groups is defined the same way, with exact sequences and filtrations interpreted appropriately (see \cite{2ab} for further details).  We will later be interested in variegated extensions of strictly commutative $2$-groups in which $Q = \mathrm B Y$ is the classifying space of a sheaf of abelian groups, and $P$ and $R$ are both sheaves of abelian groups.  In that case, we can rotate the sequence $E$ and~\eqref{eqn:62} is equivalent to a cross:
\begin{equation} \label{eqn:28} \vcenter{\xymatrix{
& & 0 \ar[d] & 0 \ar[d] \\
& 0 \ar[r] \ar[d] & P \ar[r] \ar[d] & P \ar[r] \ar[d] & 0 \\
0 \ar[r] & Y \ar[r] \ar[d] & F \ar[r] \ar[d] & W \ar[r] \ar[d] & 0 \\
0 \ar[r] & Y \ar[r] \ar[d] & R \ar[r] \ar[d] & E \ar[r] \ar[d] & 0 \\
& 0 & 0 & 0
}} \end{equation}
The reader who would prefer not to work explicitly with strictly commutative $2$-groups may translate our arguments in terms of diagram~\eqref{eqn:28}.

It is immediate from the $5$-lemma that variegated extensions form a groupoid.  The groupoid of variegated extensions of $E$ by $F$ will be denoted $\bExtPan(E,F)$.  

\setcounter{theorem}{\value{equation}}
\begin{remark} \label{rem:R=0}
	If $R = 0$, there is a canonical identification $\bExtPan(E,F) = \bExt^1(Q,P)$.  More generally, if either of the extensions $E$ or $F$ is split then $\bExtPan(E,F) = \bExt^1(Q,P)$.  For example, if $F$ is split by a morphism $R \to F$ then we obtain a morphism $R \to W$ by composition and $W/R$ is then an extension of $Q$ by $P$; the inverse operation sends an extension $X$ of $Q$ by $P$ to $X \mathop\times_Q E$.
\end{remark}

There is a canonical action of $\bExt^1(Q,P)$ on $\bExtPan(E,F)$, by Baer sum.  More precisely, given an extension~\eqref{eqn:63},
\setcounter{equation}{\value{theorem}}
\begin{equation} \label{eqn:63}
0 \to P \to X \to Q \to 0
\end{equation}
we may combine~\eqref{eqn:63} with~\eqref{eqn:62} to obtain~\eqref{eqn:64}:
\begin{equation} \label{eqn:64} \vcenter{ \xymatrix{
& 0 \ar[d] & 0 \ar[d] \\
0 \ar[r] & P \oplus P \ar@{=}[r] \ar[d] & P \oplus P \ar[r] \ar[d] & 0 \ar[d] \\
0 \ar[r] & F \oplus P \ar[r] \ar[d] & W \oplus X \ar[r] \ar[d] & Q \oplus Q \ar@{=}[d] \ar[r] & 0 \\
0 \ar[r] & R \ar[r] \ar[d] & E \oplus Q \ar[r] \ar[d] & Q \oplus Q \ar[d] \ar[r] & 0 \\
& 0 & 0 & 0
}} \end{equation}
The rows and columns are still exact.  We now push out along the codiagonal $P \oplus P \to P$ and pull back along the diagonal $Q \to Q \oplus Q$ to obtain a new variegated extension $W'$ of $E$ by $F$.  

\setcounter{theorem}{\value{equation}}
\begin{proposition}
The variegated extensions of $E$ by $F$ form a torsor under $\bExt^1(Q,P)$.
\end{proposition}
\begin{proof}
Suppose that $W$ and $W'$ are variegated extensions of $E$ by $F$.  We form their difference $W' - W$, as follows.  First, assemble the commutative diagram~\eqref{eqn:65} with exact rows and columns (the variegated extension $W \oplus W'$ of $E \oplus E$ by $F \oplus F$):
	\setcounter{equation}{\value{theorem}}
\begin{equation} \label{eqn:65} \vcenter{ \xymatrix{
& 0 \ar[d] & 0 \ar[d] \\
0 \ar[r] & P \oplus P \ar@{=}[r] \ar[d] & P \oplus P \ar[r] \ar[d] & 0 \ar[d] \\
0 \ar[r] & F \oplus F \ar[r] \ar[d] & W \oplus W' \ar[r] \ar[d] & Q \oplus Q \ar[r] \ar@{=}[d] & 0 \\
0 \ar[r] & R \oplus R \ar[r] \ar[d] & E \oplus E \ar[r] \ar[d] & Q \oplus Q \ar[r] \ar[d] & 0 \\
& 0 & 0 & 0
}} \end{equation} 
We now pull back along the diagonal map on the bottom row and push out along the \emph{difference} map on the left column.  The result is a diagram~\eqref{eqn:66}:
\begin{equation} \label{eqn:66} \vcenter{ \xymatrix{
& 0 \ar[d] & 0 \ar[d] \\
0 \ar[r] & P \ar@{=}[r] \ar@{=}[d] & P \ar[r] \ar[d] & 0 \ar[d] \\
0 \ar[r] & P \ar[r] \ar[d] & X \ar[r] \ar[d] & Q \ar@{=}[d] \ar[r] & 0 \\
 & 0 \ar[r] & Q \ar@{=}[r] \ar[d] & Q \ar[r] \ar[d] & 0 \\
& & 0 & 0
}} \end{equation}
This is simply an extension $W' - W$ of $Q$ by $P$.

This defines a morphism $\bExtPan(E,F)^2 \to \bExt^1(Q,P)$ sending $(W,W')$ to $W' - W$.  It remains to see that this construction induces an equivalence $\bExtPan(E,F)^2 \to \bExtPan(E,F) \times \bExt^1(Q,P)$ by the formula $(W,W') \mapsto (W, W'-W)$.

	It will help to think of $\bExtPan(E,F)$ as a covariantly fibered groupoid with respect to the $F$ variable.  That is, let $\bExtPan(E,-)$ be the $2$-category of diagrams~\eqref{eqn:62}, with the extension~\eqref{eqn:101} fixed.  The projection $\bExtPan(E,-) \to \bExt^1(R,-)$ from $\bExtPan(E,-)$ to the category of extensions of the fixed object $R$ by a variable object $P$ is a covariant fibration.%
	\footnote{In older terminology, this would have been called a cofibered category (or cofibered $2$-category), but that usage is fraught because there is a model structure in which these projections are fibrations.}
	
	For each pair of extensions $F$ and $F'$ of $R$ by $P$, the fiber product $F \mathop\times_R F'$ is an extension of $R$ by $P \oplus P$.  This has two projections, to $F$ and $F'$, which yield a functor:
	\begin{equation} \label{eqn:41}
		\bExtPan(E,F \mathop\times_R F') \to \bExtPan(E,F) \times \bExtPan(E,F')
	\end{equation}
	This functor is an equivalence.  Its inverse sends a pair of variegated extensions $(W,W')$ to $W \mathop\times_E W'$.

	We observe that $F \mathop\times_R F = F \mathop\times_R P^\ast$, where $P^\ast$ is the following split extension:
	\begin{equation} \label{eqn:42}
		0 \to P \to P \oplus R \to R \to 0
	\end{equation}
	We may now conclude by observing that all functors in the following chain are equivalences:
	\begin{multline} \label{eqn:43}
		\bExtPan(E,F)^2 \leftarrow \bExtPan(E,F \mathop\times_R F) \\ \simeq 
		\bExtPan(E, F \mathop\times_R P^\ast) \rightarrow \bExtPan(E,F) \times \bExtPan(E,P^\ast) \\
		\simeq \bExtPan(E,F) \times \bExt^1(Q,P)
	\end{multline}
	The last equivalence uses Remark~\ref{rem:R=0}.

%
%
\end{proof}

\begin{proposition}
Given extensions $E$ and $F$ as above, the obstruction to finding a variegated extension $E$ by $F$ is the cup product $F \cup E \in \Ext^2(Q,P)$. 
\end{proposition}
\begin{proof}
The cup product $F \cup E$ is the complex~\eqref{eqn:60}:
\begin{equation} \label{eqn:60}
0 \to P \to F \to E \to Q \to 0
\end{equation}
Splittings of this sequence are butterflies:\footnote{This is implicit in \cite[3.1]{sga7-VII}, and explicit (as well as generalized to a nonabelian context) in \cite{butterflies}, which is also the source of the term.}
\begin{equation} \label{eqn:61} \vcenter{\xymatrix{
&&& 0 \ar[dr] &&&& 0 \\
0 \ar[rr] && P \ar[rr] \ar@{=}[dd] && P \ar[dr] \ar[rr]^0 && Q \ar[rr] \ar[ur] && Q \ar@{=}[dd] \ar[rr] && 0 \\
&&&&& W \ar[ur] \ar[dr]  \\
0 \ar[rr] & & P \ar[rr] && F \ar[rr] \ar[ur] && E \ar[rr] \ar[dr] && Q \ar[rr] && 0 \\
&&& 0 \ar[ur] &&&& 0
}} \end{equation}
We leave it to the reader to rearrange this diagram into the form~\eqref{eqn:62} to make clear that $W$ is a variegated extension of $E$ by $F$.
\end{proof}

If $\varphi : \mathscr C \to \mathscr C'$ is an exact functor between abelian categories, or an exact functor between $2$-categories of strictly commutative $2$-groups over different sites, and $E$ and $F$ are exact sequences in $\mathscr C$, as above, then we have a functor:
\begin{equation} \label{eqn:29}
\bExtPan_{\mathscr C}(E,F) \to \bExtPan_{\mathscr C'}(\varphi E,\varphi F)
\end{equation}
The various manifestations of the monodromy pairing that we will encounter in subsequent sections all represent the obstructions to lifting a variegated extension of $\varphi E$ by $\varphi F$ to one of $E$ by $F$.  Provided that $\Ext^2_{\mathscr C}(E,F) = 0$, there is at least one variegated extension of $E$ by $F$ in $\bExtPan_{\mathscr C}(E,F)$ and by the choice of one, we may identify $u : \bExtPan_{\mathscr C}(E,F) \simeq \bExt^1_{\mathscr C}(Q,P)$ and $v : \bExtPan_{\mathscr C'}(\varphi E, \varphi F) \simeq \bExt^1_{\mathscr C'}(\varphi Q,\varphi P)$.  The functor $\varphi$
 induces a morphism~\eqref{eqn:30}
\begin{equation} \label{eqn:30}
\varphi : \bExt^1_{\mathscr C}(Q,P) \to \bExt^1_{\mathscr C'}(\varphi Q,\varphi P)
\end{equation}
The functor~\eqref{eqn:30} is compatible with the actions of $\bExt^1$ on $\bExtPan$, and therefore commutes with the identifications $u$ and $v$.  It follows that the obstruction to lifting a variegated extension of $\varphi E$ by $\varphi F$ lies canonically in the quotient~\eqref{eqn:31} (which we construct as a stack):
\begin{equation} \label{eqn:31}
\bExt^1_{\mathscr C'}(\varphi Q, \varphi P) \Big/ \bExt^1_{\mathscr C}(Q,P)
\end{equation}
The construction of the obstruction in~\eqref{eqn:31} depends on the choice of a variegated extension of $E$ by $F$ in $\mathscr C$.  Since the choices of this variegated extension form a torsor under $\bExt^1_{\mathscr C}(Q, P)$, the obstruction in~\eqref{eqn:31} is independent of this choice.  By descent, only the local vanishing of $\Ext^2(Q,P)$ is necessary to produce an obstruction in~\eqref{eqn:31}.

\section{The logarithmic monodromy pairing}
\label{sec:monodromy}

Let $M = [Y \to G]$ be a logarithmic $1$-motif, where $G$ is an extension of the abelian variety $A$ by the logarithmic semiabelian variety $T = \uHom(X, \logGm)$.

Let $E^{\log}$ and $F^{\log}$ be the following extensions, deduced from the weight filtration:
\begin{gather} 
0 \to T \to G \to A \to 0 \tag{$F^{\log}$} \\
0 \to G/T \to M/T \to \mathrm B Y \to 0 \tag{$E^{\log}$}
\end{gather}
Then the weight filtration gives $M$ the structure of a variegated extension of $E^{\log}$ by $F^{\log}$.

By Proposition~\ref{prop:semiabelian}, the extension $E^{\log}$ is induced from a unique extension $E^{\rm alg}$:
\begin{equation}
0 \to T^{\rm alg} \to G^{\rm alg} \to A \to 0 \tag{$E^{\rm alg}$}
\end{equation}
The extension $F^{\log}$ coincides with the extension $F^{\rm alg}$:
\begin{equation}
	0 \to A \to [A/Y] \to \mathrm BY \to 0 \tag{$F^{\rm alg}$}
\end{equation}
Denote by $E^{\rm trop}$ and $F^{\rm trop}$ the following two trivial extensions:
\begin{gather}
0 \to T/T^{\rm alg} \to G/G^{\rm alg} \to 0 \to 0 \tag{$F^{\rm trop}$} \\
0 \to 0 \to [M/G] \to \mathrm B Y \to 0 \tag{$E^{\rm trop}$}
\end{gather}
We note that $T/T^{\rm alg} = \uHom(X, \tropGm)$.  We have an exact sequence:
\begin{equation*}
0 \to \bExtPan(E^{\rm alg}, F^{\rm alg}) \to \bExtPan(E^{\rm log}, F^{\log}) \to \bExtPan(E^{\rm trop}, F^{\rm trop}) 
\end{equation*}
It follows that the obstruction to lifting $M$ to a variegated extension of $E^{\rm alg}$ by $F^{\rm alg}$ is the induced variegated extension of $E^{\rm trop}$ by $F^{\rm trop}$.  On the other hand, $\gr_0 W_\bullet M^{\rm trop} = 0$, so we have a canonical identification~\eqref{eqn:4}:
\begin{multline} \label{eqn:4}
\bExtPan(E^{\rm trop}, F^{\rm trop}) = \bExt^1(\mathrm B Y, \uHom(X, \tropGm)) \\ = \Hom(Y, \uHom(X, \tropGm)) = \Hom(X \otimes Y, \tropGm)
\end{multline}
The image of the variegated extension $M^{\rm trop}$ is therefore a pairing $X \times Y \to \tropGm$, called the \emph{logarithmic monodromy pairing}.

\setcounter{theorem}{\value{equation}}
\begin{theorem}
The monodromy pairing of a logarithmic abelian variety with constant degeneration coincides with the pairing defined in Section~\ref{sec:log-1-motifs}.
\end{theorem}
\begin{proof}
Under the identifications~\eqref{eqn:4}, the class of $M^{\rm trop}$ is the extension~\eqref{eqn:5}:
	\setcounter{equation}{\value{theorem}}
\begin{equation} \label{eqn:5}
0 \to \Hom(X, \tropGm) \to M^{\rm trop} \to \mathrm B Y \to 0
\end{equation}
Up to rotation, this is exactly the sequence induced from the pairing defined in Section~\ref{sec:log-1-motifs}.
\end{proof}

\setcounter{subsection}{\value{equation}}
\subsection{The monodromy pairing for logarithmic Jacobians}

In order to state the following corollary, we recall the definition of the \emph{tropical intersection pairing} (or \emph{edge length pairing}) of a tropical curve~\cite[12.4]{sga7-IX}.  Suppose that $\mathfrak C$ is a tropical curve, metrized by a monoid $\overnorm M$ (see \cite{logpic}).  We take $\mathbf Z^E$ to be the free abelian group of oriented edges, with the convention that $-e$ denotes $e$ with its orientation reversed.  Then the tropical intersection pairing on $\mathfrak C$ is the homomorphism
\setcounter{equation}{\value{subsection}}
\begin{equation} \label{eqn:40}
	\partial : \mathbf Z^E \to \Hom(\mathbf Z^E, \overnorm M)
\end{equation}
defined on the basis of oriented edges by the following formula:
\begin{equation} 
	\partial(e).f = \begin{cases} \ell(e) & e = f \\ -\ell(e) & e = -f \\ 0 & \text{otherwise} \end{cases}
\end{equation}
When $\mathfrak C$ is the tropicalization of a logarithmic curve $C$ over $S$, the edge $e$ corresponds to a node with a local equation $xy = t$ for some $t \in M_S$; we have $\overnorm M = \overnorm M_S$ and the length $\ell(e)$ is the image of $t$ in $\overnorm M_S$.

We sometimes abuse notation and also write $\partial$ for the map $\mathbf Z^E \times \mathbf Z^E \to \overnorm M$.  We use the same symbol also to denote the restriction of the pairing to $H_1(\mathfrak C) \subset \mathbf Z^E$.

\setcounter{theorem}{\value{equation}}
\begin{corollary} \label{cor:curves}
Suppose that $J$ is the logarithmic Jacobian of a family of proper, vertical, saturated, and logarithmically smooth logarithmic curves $C$ over $S$ with constant degeneracy.  Then the monodromy pairing of $J$ is the tropical intersection pairing of $C$.
\end{corollary}
\begin{proof}
By the proof of~\cite[Theorem~4.15.7]{logpic}, the extension~\eqref{eqn:5} is given by the tropical intersection pairing.
\end{proof}

\section{The dual logarithmic $1$-motif}
\label{sec:dual}

In this section, we will give a second construction of the monodromy pairing using the dual logarithmic $1$-motif.  This section is meant to provide comparisons to other perspectives on the monodromy pairing, and is not used in the rest of the paper.

Let $M = [G/Y]$ and $\hat M = [\hat G/\hat Y]$ be logarithmic $1$-motifs over a logarithmic scheme $S$.  A biextension of $M$ and $\hat M$ is a biextension of $G$ and $\hat G$ with trivializations over $Y \times_S \hat G$ and $G \times_S \hat Y$ that agree over $Y \times_S \hat Y$.  A biextension of $M$ and $\hat M$ by $\tropGm$ is equivalently any of the following data:
\begin{enumerate}[label=(\roman{*})]
	\item a homomorphism $M \to \bExt^1(\hat M, \logGm)$;
	\item a homomorphism $\hat M \to \bExt^1(M, \logGm)$;
	\item a homomorphism $M \otimes \hat M \to \mathrm B \logGm$ (see \cite[Section 14]{2ab}).
\end{enumerate}

If $P$ is a biextension of $M$ and $\hat M$ by $\logGm$ then it induces a biextension $\bar P$ of $M$ and $\hat M$ by $\tropGm$ via the homomorphism $\logGm \to \tropGm$.  By Lemma~\ref{lem:discrete-hom-ext}, all extensions of $G^{\rm alg}$ by $\tropGm$ are uniquely trivialized, so the biextension $\bar P$ descends to a biextension $P^{\rm trop}$ of $M^{\rm trop}$ and $\hat M^{\rm trop}$ by $\tropGm$.  We therefore obtain a commutative diagram:
\begin{equation} \label{eqn:44} \vcenter{ \xymatrix{
		M \times_S \hat M \ar[r] \ar[d] & \mathrm B \logGm \ar[d] \\
		M^{\rm trop} \times_S \hat M^{\rm trop} \ar[r] & \mathrm B \tropGm
}} \end{equation}
Since the kernel of $\mathrm B\logGm \to \mathrm B \tropGm$ is $\mathrm B\Gm$, the commutativity of~\eqref{eqn:44} implies that $P^{\rm trop}$ is the obstruction to lifting $P$ to a biextension by $\Gm$.  More specifically, if $N \to M$ and $\hat N \to \hat M$ are homomorphisms of strictly commutative $2$-groups over $S$ then trivializations of the induced biextension $P^{\rm trop} \big|_{N \times \hat N}$ are in one-to-one correspondence with lifts of $P \big|_{N \times_S \hat N}$ to biextensions of $N \times_S \hat N$ by $\Gm$.

We may make this more explicit by analyzing the data contained in the biextension of $M^{\rm trop}$ and $\hat M^{\rm trop}$ by $\tropGm$.  Let us present $M^{\rm trop}$ as $[ \uHom(X,\tropGm)/Y ]$ and $\hat M^{\rm trop}$ as $[\uHom(\hat X,\tropGm)/\hat Y]$.  Then the biextension  by $\tropGm$ is a biextension of $\uHom(X,\tropGm)$ and $\uHom(\hat X,\tropGm)$ by $\tropGm$ with trivializations along $\uHom(X, \tropGm) \times_S \hat Y$ and along $Y \times_S \uHom(\hat X, \tropGm)$ that agree on $Y \times_S \hat Y$.

By Proposition~\ref{prop:tropical-hom-ext}, every extension of $\uHom(\hat X, \tropGm)$ by $\tropGm$ is trivial.  This gives us the identification~\eqref{eqn:51}, using Proposition~\ref{prop:tropical-hom-ext} again:
\begin{equation} \label{eqn:51}
	\bExt^1\bigl( \:\uHom(\hat X, \tropGm), \tropGm \bigr) = \mathrm B \: \uHom\bigl( \: \uHom( \hat X, \tropGm ), \tropGm \bigr)  = \mathrm B \hat X
\end{equation}
Thus a biextension of $\uHom(X, \tropGm)$ and $\uHom(\hat X, \tropGm)$ by $\tropGm$ is specified equivalently by an extension of $\uHom(X,\tropGm)$ by $\hat X$, hence is uniquely trivialized by Proposition~\ref{prop:discrete-ext}.

Thus the biextension is encoded by pairings $\uHom(X, \tropGm) \times \hat Y \to \tropGm$ and $Y \times \uHom(X, \tropGm) \to \tropGm$ that agree on $Y \times \hat Y$.  In other words, it is given by maps $\hat Y \to X$ and $Y \to \hat X$ such that the two pairings on $Y \times \hat Y$ induced from $Y \times \hat Y \to Y \times X$ and $Y \times \hat Y \to \hat X \times \hat Y$ agree.

We will apply this next to the dual logarithmic $1$-motif, $\hat M = \bExt^1(M, \logGm)$.  We start by verifying that $\bExt^1(M, \logGm)$ is indeed a logarithmic $1$-motif:

\setcounter{theorem}{\value{equation}}
\begin{lemma} \label{lem:dual-1-motif}
	If $M$ is a logarithmic $1$-motif then so is $\bExt^1(M, \logGm)$.
\end{lemma}
\begin{proof}
	Let $M = [G/Y]$ be a logarithmic $1$-motif over $S$ with $G$ a logarithmic semiabelian variety whose torus part is $T = \uHom(X,\logGm)$ and whose abelian part is $A$.  By Lemma~\ref{lem:discrete-hom-ext}, we have $\bExt^1(A, \logGm) = \bExt^1(A,\Gm)$: the dual of an abelian variety as a logarithmic $1$-motif coincides with its dual as an abelian variety.  Noting that $\uHom(T, \logGm) = X$ by Proposition~\ref{prop:log-hom-ext}, we obtain $\bExt^1(T, \logGm) = \mathrm B X$.

	Dualizing the exact sequence~\eqref{eqn:52} we obtain~\eqref{eqn:53}:
	\setcounter{equation}{\value{theorem}}
	\begin{gather}
		0 \to Y \to A \to [M/T] \to 0 \label{eqn:52} \\
		0 \to \bExt^1([M/T], \logGm) \to \bExt^1(A, \Gm) \to \bExt^1(Y, \logGm) \label{eqn:53}
	\end{gather}
	As $\uExt^1(Y, \logGm) = 0$ (since $Y$ is locally free), the sequence can be rotated to give~\eqref{eqn:54}, which shows that $\hat G = \bExt^1([M/T],\logGm)$ is a logarithmic semiabelian variety:
	\begin{equation} \label{eqn:54}
		0 \to \uHom(Y, \logGm) \to \bExt^1([M/T], \logGm) \to \bExt^1(A, \Gm) \to 0
	\end{equation}
	Dualizing~\eqref{eqn:55} gives~\eqref{eqn:56}:
	\begin{gather} 
		0 \to T \to M \to [M/T] \to 0 \label{eqn:55} \\
		0 \to \bExt^1([M/T], \logGm) \to \bExt^1(M, \logGm) \to \bExt^1(T, \logGm) \label{eqn:56}
	\end{gather}
	We have $\uExt^1(T, \logGm) = 0$ by Proposition~\ref{prop:log-hom-ext}.  We may therefore rotate~\eqref{eqn:56} to give~\eqref{eqn:57}:
	\begin{equation} \label{eqn:57}
		0 \to X \to \bExt^1([M/T], \logGm) \to \bExt^1(M, \logGm) \to 0
	\end{equation}
	This shows that $\hat M = \bExt^1(M, \logGm) = [\hat G/X]$ is a logarithmic $1$-motif with torus part $\hat T = \uHom(Y, \logGm)$, semiabelian part $\hat G = \bExt^1([M/T], \logGm)$, and abelian part $\hat A = \bExt^1(A, \Gm)$.
\end{proof}

\setcounter{theorem}{\value{equation}}
\begin{remark}
	The dual logarithmic~$1$-motif of Lemma~\ref{lem:dual-1-motif} coincides with the dual logarithmic~$1$-motif of \cite[Section~2.7]{KKN2}.  We recall the construction of op.\ cit., in our notation.  The dual logarithmic $1$-motif was defined as follows: first $\hat G^{\rm alg}$ is constructed as the semiabelian variety of $Y$-split extensions of $A$ by $\Gm$, where $Y \to A$ is the composition $Y \to G \to A$ \cite[Section~2.7.1]{KKN2}.  We may identify these with extensions of $[M/T]=[A/Y]$ by $\Gm$.  Then $\hat G$ is constructed as the associated logarithmic semiabelian variety~\cite[Section~2.7.2]{KKN2}, which by Proposition~\ref{prop:semiabelian} coincides with the logarithmic semiabelian variety of $Y$-split extensions of $A$ by $\logGm$, or equivalently, extensions of $[M/T]=[A/Y]$ by $\logGm$.  Finally, $X \to \hat G$ sends each $x \in X$ to the $Y$-split extension associated with it by pushout along $x : T \to \logGm$, which is exactly the same as the extension of $[M/T]=[A/Y]$ induced along the map $X \to \bExt^1([M/T], \logGm)$ in~\eqref{eqn:57}.
\end{remark}

Let $M = [G/Y]$ be a logarithmic $1$-motif with torus part $\uHom(X,\logGm)$ and let $\hat M$ be its dual.  In the notation preceding the lemma, we have $\hat X = Y$ and $\hat Y = X$.  By the discussion preceding the lemma, the biextension realizing the duality between $M$ and $\hat M$ is given by homomorphisms $Y \to \hat X = Y$ and $X = \hat Y \to X$, namely the identity morphisms, that restrict to the same pairing on $Y \times \hat Y = Y \times X$.

We have the following immediate consequence of the construction:

\begin{proposition}
	The pairing on $Y \times X$ induced from the duality between $M^{\rm trop}$ and $\hat M^{\rm trop}$ coincides with the monodromy pairing.
\end{proposition}

\setcounter{subsection}{\value{theorem}}
\subsection{Component groups}
\label{sec:comp-grps}

In a slightly more concrete situation, this obstruction can be made more explicit.  Assume that $P$ is a biextension of $M$ and $\hat M$ by $\logGm$ over $S$, and that $N$ and $\hat N$ are smooth commutative group schemes over $S$ with component groups $\Phi$ and $\hat \Phi$.  Assume that homomorphisms $N \to M$ and $\hat N \to \hat M$ have been specified.  The biextension $P^{\rm trop}$ of $M^{\rm trop}$ and $\hat M^{\rm trop}$ by $\tropGm$ restricts to a biextension of $N$ and $\hat N$, and this descends by Lemma~\ref{lem:discrete-hom-ext} to a biextension of $\Phi$ and $\hat\Phi$ by $\tropGm$ that we denote $Q$.

The universal example of this construction is to take $N$ and $\hat N$ to be the preimages of the maximal \'etale subgroups of $M^{\rm trop}$ and $\hat M^{\rm trop}$.  That is, let $\Phi$ and $\hat\Phi$ be the restrictions of $M^{\rm trop}$ and $\hat M^{\rm trop}$ to the small, strict \'etale site of $S$ and use the same symbols to denote their espaces \'etal\'es.  Let $N$ and $\hat N$ be their preimages in $M$ and $\hat M$.  Based on Lemma~\ref{lem:neron-model}, below, it would be reasonable to call these the N\'eron models of $M$ and $\hat M$.  (See \cite{HMOP} for much more about this perspective on N\'eron models.)

The exact sequence~\eqref{eqn:69} induces a long exact sequence~\eqref{eqn:70}:
\setcounter{equation}{\value{subsection}}
\begin{gather}
	0 \to \tropGm \to \mathbf Q \otimes \tropGm \to \mathbf Q/\mathbf Z \otimes \tropGm \to 0 \label{eqn:69} \\
	\begin{split}
		\uBilin(\Phi,\hat\Phi;\mathbf Q \otimes \tropGm) \to \uBilin(\Phi,\hat\Phi; \mathbf Q/\mathbf Z \otimes \tropGm)  \to & \uBiext(\Phi,\hat\Phi;\tropGm) \\ & \to \uBiext(\Phi,\hat\Phi;\mathbf Q \otimes \tropGm) 
	\end{split}
	\label{eqn:70} 
\end{gather}
If there is an integer $n > 0$ such that $n \Phi = 0$ or $n \hat\Phi = 0$ then all bilinear maps $\Phi \times \hat\Phi \to \mathbf Q \otimes \tropGm$ and all biextensions of $\Phi$ and $\hat\Phi$ by $\mathbf Q$ vanish.  Therefore the biextension $Q$ is induced from a unique bilinear pairing $\Phi \times \hat\Phi \to \mathbf Q/\mathbf Z \otimes \tropGm$, and this bilinear pairing obstructs the existence of a biextension of $N$ and $\hat N$ by $\Gm$ inducing $P \big|_{N \times \hat N}$.

\setcounter{subsection}{\value{equation}}
\subsection{Logarithmic structures of rank~$1$}
\label{sec:rank-1}

Let us specialize further to the case where $\overnorm M_S = \mathbf N$.  For example, $S$ might be the closed point of a discrete valuation ring, with its standard logarithmic structure.

The following proposition is analogous to \cite[Th\'eor\`eme~11.5]{sga7-IX}, but is essentially trivial from the tropical perspective:

\setcounter{theorem}{\value{subsection}}
\begin{proposition} \label{prop:neron-Z}
	Let $M^{\rm trop} = [ \uHom(X, \tropGm) / Y ]$ be a tropical $1$-motif and let $\Phi$ be defined as above.  Then $\Phi = [ \uHom(X,\mathbf Z) / Y ]$.
\end{proposition}
\begin{proof}
	By assumption, we have $\overnorm M_S^{\rm gp} = \mathbf Z$, so this is only a matter of interpreting the definition.
\end{proof}

\begin{corollary}
	Let $M^{\rm trop}$ and $\hat M^{\rm trop}$ be dual tropical $1$-motifs and let $\Phi$ and $\hat\Phi$ be defined as above.  Then the restriction of the duality biextension from $M^{\rm trop} \times \hat M^{\rm trop}$ to $\Phi \times \hat\Phi$ puts $\Phi$ and $\hat\Phi$ into duality.
\end{corollary}
\begin{proof}
	By Proposition~\ref{prop:neron-Z}, we have an exact sequence~\eqref{eqn:71}, which dualizes to~\eqref{eqn:72}:
	\setcounter{equation}{\value{theorem}}
	\begin{gather}
		0 \to Y \to \uHom(X,\mathbf Z) \to \Phi \to 0 \label{eqn:71}  \\
		0 \to X \to \uHom(Y,\mathbf Z) \to \bExt^1(\Phi, \mathbf Z) \to 0 \label{eqn:72} 
	\end{gather}
	Since we also have $\hat\Phi = [ \uHom(Y, \mathbf Z) / X ]$, we conclude that $\bExt^1(\Phi,\mathbf Z) = \hat\Phi$.
\end{proof}

By the discussion at the end of Section~\ref{sec:comp-grps}, this duality is equivalently a perfect pairing:
\begin{equation*}
	\Phi \times \hat\Phi \to \mathbf Q/ \mathbf Z
\end{equation*}

\setcounter{subsection}{\value{equation}}
\subsection{N\'eron models}
\label{sec:neron}

In order to relate the discussion in Sections~\ref{sec:comp-grps} and~\ref{sec:rank-1} to \cite{sga7-VIII}, \cite{sga7-IX}, and \cite{Gillibert}, we have to consider a degenerating family of abelian varieties, so we must temporarily relax the assumption that our logarithmic $1$-motifs have constant degeneracy.  Instead, we will assume that $S$ is the spectrum of a discrete valuation ring and that $M$ and $\hat M$ are logarithmic abelian varieties extending abelian varieties $M_\eta$ and $\hat M_\eta$ over the general fiber.  We write $N$ and $\hat N$ for the N\'eron models of $N_\eta = M_\eta$ and $\hat N_\eta = \hat M_\eta$.

Let $\Phi$ denote the restriction of $M^{\rm trop}$ to the small (strict) \'etale site of $S$.  We use the same symbol for the extension of $\Phi$ to the strict \'etale site of all logarithmic schemes over $S$.  Then $\Phi$ comes with a canonical map to $M^{\rm trop}$.

\setcounter{theorem}{\value{subsection}}
\begin{lemma} \label{lem:neron-model}
	The preimage of $\Phi$ along $M \to M^{\rm trop}$ is the N\'eron model of $M_\eta$.  Hence $\Phi$ is the component group of the N\'eron model of $M$.
\end{lemma}
\begin{proof}
	Write $N$ for the preimage of $\Phi$ in $M$.  We must show that $N$ is smooth and separated and has the extension property for \'etale points \cite[\S1.2, Criterion 9]{BLR}.  Logarithmic smoothness and separatedness are inherited from $M$.  But $\Phi$ is strict over $S$, so $N$ is also strict over $S$, which, together with its logarithmic smoothness, implies it is smooth over $S$.

	Finally, suppose that $S'$ is the spectrum of a discrete valuation ring and is strict and \'etale over $S$.  Then a map $S'_\eta \to M$ extends uniquely to $S' \to M$ by the properness of $M$ over $S$ \cite[Proposition~11.3]{KKN4}.  But $S'$ is strict and \'etale over $S$, so the induced map $S' \to M^{\rm trop}$ must factor through $\Phi$.  Therefore $S' \to M$ factors uniquely through $N$, as required.
\end{proof}

We define $\hat\Phi$ by applying the same construction to $\hat M$.  As illustrated in~\eqref{eqn:44}, the biextension $\Phi \times \hat\Phi \to \mathrm B\tropGm$ obstructs lifts of the canonical biextension $M \times \hat M \to \mathrm B\logGm$ to $\mathrm B \Gm$.

Since $\Phi$ and $\hat \Phi$ are \'etale over $S$ and trivial over $\eta$, the biextension $\Phi \times \hat\Phi \to \mathrm B \tropGm$ is uniquely determined by its restriction to the small \'etale site of the closed fiber.  This places us in the situation of Section~\ref{sec:rank-1}, where it is equivalent to a perfect pairing:
\begin{equation*}
	\Phi \times \hat\Phi \to \mathbf Q/\mathbf Z \otimes \overnorm M_S^{\rm gp} = \mathbf Q/\mathbf Z
\end{equation*}

This recovers \cite[Th\'eor\`eme~7.2~b)]{sga7-VIII} and proves \cite[Conjecture~1.3]{sga7-IX} for dual abelian varieties over the generic point of a discrete valuation ring that can be extended to dual logaritihmic abelian varieties.

We also note that $\mathbf Q/\mathbf Z \otimes \tropGm$ is trivial in the Kummer logarithmic flat topology (see Proposition~\ref{prop:divisible}), so we obtain a canonical biextension of $N$ and $\hat N$ by $\Gm$ in the Kummer logarithmic flat topology.  This gives \cite[Th\'eor\`eme~4.1.1 and Proposition~4.2.1]{Gillibert}, again for dual abelian varieties that can be extended to dual logarithmic abelian varieties.

\section{The torsion realization}
\label{sec:torsion}

In this section, we will have to go back and forth between the logarithmic flat topology and the strict flat topology.  The strict flat topology is generated by strict fppf covers.  The logarithmic flat topology can refer either to the Kummer logarithmic flat topology, or to a finer full logarithmic flat topology in which the covers are finitely presented, logarithmic flat universal surjections (all of these properties are stable under fine and saturated base change and under composition \cite[Corollary~4.12]{Olsson03}, so this does indeed form a topology).  The generating covers of the Kummer logarithmic flat topology have the additional requirement of being Kummer.

It will be important that logarithmic $1$-motifs are divisible in the logarithmic flat topology, but not divisible in the strict flat topology.

\begin{proposition} \label{prop:divisible}
Logarithmic $1$-motifs are locally divisible in the logarithmic flat topology on the base.
\end{proposition}
\begin{proof}
A logarithmic $1$-motif is a quotient of an extension of an abelian variety by a logarithmic torus, $\Hom(X, \logGm)$, which is an extension of $\Hom(X, \tropGm)$ by an algebraic torus.  Quotients of divisible groups are divisible, and abelian varieties and algebraic tori are locally divisible in the strict flat topology, so it remains to see that $\Hom(X, \tropGm)$ is locally divisible in the logarithmic flat topology.  This follows from the fact that Kummer extensions are covers in the logarithmic flat topology: a homomorphism $X \to \overnorm M_S^{\rm gp}$ is the $n$-th multiple of a homomorphism $X \to \frac{1}{n} \overnorm M_S^{\rm gp}$ and the logarithmic schemes $f : S' \to S$ with a factorization of $f^{-1} \overnorm M_S^{\rm gp} \to \overnorm M_{S'}^{\rm gp}$ through $\frac{1}{n} f^{-1} \overnorm M_S^{\rm gp}$ form a cover in the logarithmic flat topology.
\end{proof}

Passing to $n$-torsion subgroups is exact for divisible groups, so the filtration~\eqref{eqn:3} induces a filtration~\eqref{eqn:6}
\begin{align} 
W_{-1} M[n] & = T[n] = T^{\rm alg}[n] = \Hom(X, \mu_n) \notag \\
W_0 M[n] & = G[n] = G^{\rm alg}[n] \label{eqn:6} \\
W_1 M[n] & = M[n] \notag
\end{align}
with the following graded pieces:
\begin{align*}
\gr_{-1} W_\bullet M [n] & = \Hom(X, \mu_n) \\
\gr_0 W_\bullet M [n] & = A[n] \\
\gr_1 W_\bullet M [n] & = n^{-1} Y / Y \simeq Y/nY
\end{align*}
We introduce a subscript $n$ to denote the cokernel of the multiplication by $n$ map, so that $\gr_1 W_\bullet M[n] = Y_n$.

Furthermore, $M[n]$ has the structure of a variegated extension of $E[n]$ by $F[n]$:
\begin{gather}
	0 \to T[n] \to G[n] \to A[n] \to 0 \tag{$F[n]$} \label{eqn:104} \\
	0 \to A[n] \to M[n]/T[n] \to Y_n \to 0 \tag{$E[n]$} \label{eqn:103}
\end{gather}
Note that, in contrast to the filtration~\eqref{eqn:3}, the filtration~\eqref{eqn:6} is an honest filtration, in the sense that $W_i M \subset W_{i+1} M$ for all $i$.  

Let $\Lambda$ be the ring $\mathbf Z/n\mathbf Z$.  Both of the exact sequences $E[n]$ and $F[n]$ are pulled back from exact sequences on the strict flat site of $S$.  Since $\uExt^2_{\fppf,\Lambda}(Y_n, \Hom(X,\mu_n)) = 0$, there is at least one variegated extension of $E[n]$ by $F[n]$ locally on the strict flat site of $S$, and therefore the obstruction to descending $M[n]$ to the strict flat site is its class in the quotient:
\begin{multline} \label{eqn:7}
\Omega = \bExt^1_{\logfppf,\Lambda}(Y_n, \uHom(X_n, \mu_n)) \Big/ \bExt^1_{\fppf,\Lambda}(Y_n, \uHom(X_n, \mu_n)) \\
= \bExt^1_{\logfppf,\Lambda}(X_n \otimes Y_n, \mu_n) \Big/ \bExt^1_{\fppf,\Lambda}(X_n \otimes Y_n, \mu_n)
\end{multline}

To calculate the class of the variegated extension $M[n]$, we relate it to the variegated extension $A$ using the following commutative diagram:
\begin{equation*} \xymatrix{
0 \ar[d] & 0 \ar[d] & 0 \ar[d] \\
\uHom(X \otimes Y, \Gm) \ar[r]^-{[n]} \ar[d] & \uHom(X \otimes Y, \Gm) \ar[r] \ar[d] & \bExt^1_{\fppf,\Lambda}(X_n \otimes Y_n, \mu_n) \ar[r] \ar[d] & 0 \\
\uHom(X \otimes Y, \logGm) \ar[r]^-{[n]} \ar[d] & \uHom(X \otimes Y, \logGm) \ar[r] \ar[d] & \bExt^1_{\Lambda,\logfppf}(X_n \otimes Y_n, \mu_n) \ar[r]  \ar[d] & 0 \\
\uHom(X \otimes Y, \tropGm) \ar[r]^-{[n]} \ar[d] & \uHom(X \otimes Y, \tropGm)  \ar[d] \ar@{-->}[r] & \Omega \ar[d] \ar[r] & 0 \\
0 & 0 & 0
} \end{equation*}
All of the entries in the diagram may be regarded as sheaves and stacks in the strict flat site of $S$.  In particular, we compute the quotients $\tropGm = \logGm/\Gm$ in the \emph{strict} flat site.

The diagram allows us to identify the obstruction group~\eqref{eqn:7} with the quotient~\eqref{eqn:8}:
\begin{equation} \label{eqn:8}
\uHom\bigl(X \otimes Y, \tropGm \bigr)_n =  \uHom\bigl(X \otimes Y, (\tropGm)_n \bigr)
\end{equation}
Moreover, the map~\eqref{eqn:12}
\begin{equation} \label{eqn:12}
\uHom(X \otimes Y, \tropGm) \to \uHom \bigl(X \otimes Y, (\tropGm)_n \bigr)
\end{equation}
is reduction modulo $n$, so the obstruction in~\eqref{eqn:7} is the reduction of the logarithmic monodromy pairing modulo $n$.  We have proved the following proposition:

\begin{proposition}
Under the identification of~\eqref{eqn:7} and~\eqref{eqn:8}, the class of the variegated extension $M[n]$ is the reduction of the logarithmic monodromy pairing modulo $n$.
\end{proposition}

Completing at a prime $\ell$, we obtain an obstruction to descending the variegated extension of $E_\ell$ by $F_\ell$ to the strict flat site in~\eqref{eqn:9}:
\begin{equation} \label{eqn:9}
\bExt^1_{\logfppf,\mathbf Z_\ell}(X_\ell \otimes Y_\ell, \mu_{\ell^\infty}) \Big/ \bExt^1_{\fppf,\mathbf Z_\ell}(X_\ell \otimes Y_\ell, \mu_{\ell^\infty}) 
= \uHom\bigl(X_\ell \otimes Y_\ell, (\tropGm)_\ell \bigr)
\end{equation}

In the case where $M$ is the special fiber of the N\'eron model of an abelian variety over a discrete valuation ring, the class of $M_\ell$ was Grothendieck's construction of the monodromy pairing \cite[\S9.6]{sga7-IX}.  Since abelian varieties with semistable reduction over the generic point of a discrete valuation ring extend to logarithmic abelian varieties \cite[Corollary~3.5]{KKN6}, we obtain the following corollary:

\begin{corollary}[Grothendieck]
	The monodromy pairing, as defined in \cite[\S9.6]{sga7-IX}, is integrally defined and independent of $\ell$.
\end{corollary}
\begin{proof}
	The monodromy pairing coincides with the completion of the class of $M$ at $\ell$ and the class of $M$ is visibly integral and independent of $\ell$.
\end{proof}

\section{The \'etale realization}
\label{sec:etale}

We reconsider the calculation of Section~\ref{sec:torsion} under the additional assumption that $n$ is invertible in $\mathcal O_S$.  We may now work in the logarithmic \'etale and the strict \'etale topologies instead of the logarithmic flat and strict flat topologies.  In this section, the logarithmic \'etale topology may refer either to the full logarithmic \'etale topology or to the Kummer logarithmic \'etale topology.

Since $X_n$ and $Y_n$ are locally free $\Lambda$-modules, any extension of $X_n \otimes Y_n$, in either the strict or logarithmic \'etale topology is determined entirely by the action of $\pi_1^{\et}(S)$ or $\pi_1^{\loget}(S)$, respectively.  The obstruction to descending a $\pi_1^{\loget}(S)$ action to a $\pi_1^{\et}(S)$-action is the action of the logarithmic inertia group, $I$, which fits into a locally split exact sequence:
\begin{equation*}
1 \to I \to \pi_1^{\loget}(S) \to \pi_1^{\et}(S) \to 1
\end{equation*}
The logarithmic inertia group operates trivially on $X_n$, $Y_n$, and $\mu_n$, so the action of $I$ on an extension of $X_n \otimes Y_n$ by $\mu_n$ is given by a homomorphism~\eqref{eqn:10}:
\begin{equation} \label{eqn:10}
I \to \Hom(X_n \otimes Y_n, \mu_n)
\end{equation}
Any such homomorphism factors through the maximal $n$-torsion quotient of $I$:
\begin{equation} \label{eqn:73}
I_n \simeq \Hom(n^{-1} \overnorm M_S^{\rm gp} / \overnorm M_S^{\rm gp}, \mu_n) \simeq \Hom(\overnorm M_S^{\rm gp}, \mu_n)
\end{equation}
Therefore the homomorphism~\eqref{eqn:10} lies in
\begin{equation} \label{eqn:11}
\Hom \bigl( \uHom(\overnorm M_S^{\rm gp}, \mu_n), \uHom(X_n \otimes Y_n, \mu_n) \bigr) = \Hom \bigl( X_n \otimes Y_n, (\tropGm)_n \bigr)
\end{equation}

In order to proceed, we must calculate the action of $I_n = \Hom(\overnorm M_S^{\rm gp}, \mu_n)$ on $\logGm(T)$ when $T$ is finite and Kummer \'etale over $S$.  Suppose that $\alpha \in \logGm(T) = \Gamma(T, M_T^{\rm gp})$ and $\alpha^n$ lies in the logarithmic structure of $T$ pulled back from $S$.  Then, according to the second identification of~\eqref{eqn:73}, an element $g \in \Hom(\overnorm M_S^{\rm gp}, \mu_n)$ acts on $\alpha$ by the formula
\begin{equation}\label{eqn:32}
g.\alpha = g(n \overnorm\alpha) \alpha
\end{equation}
where $\overnorm\alpha$ is the image of $\alpha$ in $\overnorm M_T^{\rm gp}$.

\setcounter{theorem}{\value{equation}}
\begin{proposition}
Under the identifications~\eqref{eqn:73} and~\eqref{eqn:11}, the image of the variegated extension $M[n]$ is the reduction of the pairing constructed in~\eqref{eqn:4}.
\end{proposition}
\begin{proof}
We wish to show that the identifications of the obstruction group~\eqref{eqn:7} with homomorphisms~\eqref{eqn:11} given in this section and in the last coincide.  The assertion is local in the strict \'etale topology, so we may assume that $X_n$, $Y_n$, and $\mu_n$ are all constant sheaves on $S$.  We now have a commutative diagram with exact columns:
\begin{equation*} \vcenter{ \xymatrix{
0 \ar[d] &  & 0 \ar[d] \\
\uHom(X \otimes Y, \Gm) \ar[r] \ar[d] & \uExt^1_{\et,\Lambda}(X_n \otimes Y_n, \mu_n) \ar@{=}[r] \ar[d] & \uHom(\pi_1^{\et}(S), \uHom(X_n \otimes Y_n, \mu_n)) \ar[d] \\
\uHom(X \otimes Y, \logGm)  \ar[r] \ar[d] & \uExt^1_{\loget,\Lambda}(X_n \otimes Y_n, \mu_n) \ar@{=}[r]  & \uHom(\pi_1^{\loget}(S), \uHom(X_n \otimes Y_n, \mu_n)) \ar[d] \\
\uHom(X \otimes Y, \tropGm)  \ar[d] \ar@{-->}[rr] & & \uHom(I, \uHom(X_n \otimes Y_n, \mu_n)) \ar[d] \\
0 &  & 0
}} \end{equation*}
	All of the entries in the diagram are sheaves in the small, strict \'etale site of $S$.  We choose a local lift $\tilde{\partial}$ of the logarithmic monodromy pairing $\partial$ to $\Hom(X \otimes Y, \logGm)$.  Then the image of $\tilde{\partial}$ in $\Hom(\pi_1^{\loget}(S), \uHom(X \otimes Y, \mu_n) )$ is the obstruction to lifting this pairing along $[n] : \logGm \to \logGm$.  We can choose such a lift $\alpha$ locally in the logarithmic \'etale topology, and the action of $g \in I_n = \Hom(\overnorm M_S^{\rm gp}, \mu_n)$ on $\alpha : X \otimes Y \to M_{S'}^{\rm gp}$ is by~\eqref{eqn:32}.
\setcounter{equation}{\value{theorem}}
Note that $\overnorm\alpha \in n^{-1} \overnorm M_S^{\rm gp} \cap \overnorm M_{S'}^{\rm gp}$.  But $\overnorm\alpha = n^{-1} \partial$, by assumption, so the homomorphism $X_n \otimes Y_n \to (\tropGm)_n$ corresponding by~\eqref{eqn:11} to the image of $\partial$ is the reduction of $\partial$ modulo~$n$.
\end{proof}

\section{The Betti realization}
\label{sec:betti}

In this section, we will apply the realization functor of \cite{KKN-pic,KKN1}.  All logarithmic schemes will be of finite type over $\mathbf C$, and all morphisms will be morphisms over $\mathbf C$; we continue to assume that logarithmic schemes are fine and saturated, except we will have need of one non-integral logarithmic structure to construct the Kato--Nakayama space.

We write $S^{\rm KN}$ for the Kato--Nakayama space of a logarithmic scheme $S$.  Following \cite[(1.2)]{KN}, one may construct it by evaluating $\Hom(P,S)$, where $P$ is the scheme $\Spec \mathbf C$, with the logarithmic structure $\mathbf R_{\geq 0} \times S^1(1)$ (where $S^1(1) = \mathbf R(1) / \mathbf Z(1) = i \mathbf R / 2 \pi i \mathbf Z$),%
\footnote{According to our conventions, $P$ is a scheme with a logarithmic structure, but not a logarithmic scheme, because its logarithmic structure is not integral.}
the homomorphism
\begin{equation*}
	\varepsilon : \mathbf R_{\geq 0} \times S^1(1) \to \mathbf C
\end{equation*}
given by $\varepsilon(\lambda, \mu) = \lambda \mu$.  The topology on $S^{\rm KN}$ is the weakest so that every local section of $\mathcal O_S$ determines a continuous map to $\mathbf C$ and every local section of $M_S^{\rm gp}$ determines a continuous map to $\mathbf R_{\geq 0} \times S^1(1)$.

Similarly, we write $S^{\rm an}$ for the analytification of the underlying scheme of $S$.  This can be constructed by the same procedure as above.  Its point set is $\Hom(\Spec \mathbf C, S)$, topologized in the coarsest way so that any local section of $\mathcal O_S$ induces a continuous function to $\mathbf C$.  There is a projection $S^{\KN} \to S^{\rm an}$ by restricting a logarithmic morphism $P \to S$ to its underlying morphism of schemes.

While a logarithmic semiabelian variety $G$ over a logarithmic scheme $S$ is not a logarithmic scheme, it still has a well-defined Kato--Nakayama space, $G^{\rm KN}$.  The underlying set of $G^{\rm KN}$ is the set of maps $\Hom(P,G)$, as above.  This set is given the finest topology so that for all logarithmic schemes $U$ over $S$, all maps $U \to G$ induce continuous maps $U^{\rm KN} \to G^{\rm KN}$.  This allows us to form the Kato--Nakayama space of a logarithmic $1$-motif $M = [G/Y]$, in which $G$ is a logarithmic semiabelian variety over $S$ and $Y$ is a lattice over $S$, as the topological stack $[G^{\rm KN}/Y^{\rm KN}]$, where $Y^{\rm KN}$ is the Kato--Nakayama space of the espace \'etal\'e of $Y$.

\begin{example} \label{ex:log-sav-KN}
The associated group of $\mathbf R_{\geq 0} \times S^1(1)$ is $S^1(1)$, so the Kato--Nakayama space of $\logGm$ is $S^1(1)$.  This implies that the map $\Gm^{\KN} \to \logGm^{\KN}$ is surjective, with kernel $\mathbf R_{> 0}$.
\end{example}

The following proposition generalizes the example:

\begin{proposition}
	If $G$ is a logarithmic semiabelian variety over $S$ and $Q$ is a $G$-torsor over $S$ then $Q^{\rm KN}$ is proper over $S^{\rm KN}$.
\end{proposition}
\begin{proof}
	Since $G$ is an extension of an abelian variety $A$ by a logarithmic torus $T$, its Kato--Nakayama space $G^{\rm KN}$ is an extension of $A^{\rm KN}$ by $T^{\rm KN}$.  The Kato--Nakayama space $A^{\rm KN}$ is the pullback of the analytification $A^{\rm an}$ of $A$, which is proper over $S^{\rm an}$, along the map $S^{\rm KN} \to S^{\rm an}$.  Replacing $S$ with $A$, we can therefore assume that $G = T$ is a logarithmic torus and that $Q \simeq T$.  Working locally, we can also assume that $T$ is split, and treating one factor at a time, we can assume that $T = \logGm$.  We are in the situation of Example~\ref{ex:log-sav-KN}, where $T^{\rm KN} = S^1(1) \times S^{\rm KN}$, which is certainly proper over $S^{\rm KN}$.
\end{proof}

\begin{proposition} \label{prop:surj-kn}
	If $f : X \to Y$ is a finite type universal surjection of logarithmic schemes then the induced morphism of Kato--Nakayama spaces $X^{\rm KN} \to Y^{\rm KN}$ is surjective.
\end{proposition}
\begin{proof}
	Since formation of the Kato--Nakayama space commutes with strict base change, we may assume that the underlying scheme of $Y$ is $\Spec \mathbf C$.  Then $Y^{\rm KN} = \Hom(\overnorm M_Y^{\rm gp}, S^1(1))$.  Since $X$ is of finite type, there are finitely many distinct monoid homomorphisms $\overnorm M_Y \to \overnorm M_{X,x}$, as $x$ ranges over geometric points of $X$.  We argue that at least one of these must be injective.  It will then follow that 
	\begin{equation*}
		x^{\rm KN} = \Hom(\overnorm M_{X,x}^{\rm gp}, S^1(1)) \to \Hom(\overnorm M_Y^{\rm gp}, S^1(1)) = Y^{\rm KN}
	\end{equation*}
	is surjective (since $S^1(1)$ is divisible) and therefore that $X^{\rm KN} \to Y^{\rm KN}$ is surjective, as required.

	Suppose that $\varphi : \overnorm M_Y^{\rm gp} \to \mathbf Q$ is any homomorphism.  Let $Y'$ have the same underlying scheme as $Y$, but give $Y'$ a logarithmic structure whose characteristic monoid is the saturation of $\varphi(\overnorm M_Y)$ (which is isomorphic to $\mathbf N$) and such that the induced map $\overnorm M_Y^{\rm gp} \to \overnorm M_{Y'}^{\rm gp}$ is $\varphi$.  Let $X'$ be the base change of $X$ to $Y'$.  Then $X'$ is nonempty because $f$ was assumed universally surjective, so we may choose a geometric point $x'$ of $X'$.  Since $\overnorm M_{Y'}^{\rm gp} \simeq \mathbf N$, the map $\overnorm M_{Y'} \to \overnorm M_{X',x'}$ must be injective.  Therefore the inclusion homomorphism $\overnorm M_{Y'}^{\rm gp} \to \mathbf Q$ extends to $\overnorm M_{X',x'}^{\rm gp}$.  In particular, every homomorphism $\varphi : \overnorm M_Y^{\rm gp} \to \mathbf Q$ extends to $\overnorm M_{X,x}^{\rm gp} \to \mathbf Q$ for some $x \in X$.  

	It follows that the maps of rational vector spaces $\Hom(\overnorm M_{X,x}^{\rm gp}, \mathbf Q) \to \Hom(\overnorm M_Y^{\rm gp}, \mathbf Q)$ are jointly surjective.  But there are only finitely many distinct vector spaces $\Hom(\overnorm M_{X,x}^{\rm gp}, \mathbf Q)$, so at least one of these must be surjective.  The corresponding map $\overnorm M_Y^{\rm gp} \to \overnorm M_{X,x}^{\rm gp}$ must therefore be injective.
\end{proof}

\begin{remark}
	I thank the referee for bringing to my attention that some finiteness hypothesis is necessary in Proposition~\ref{prop:surj-kn}, if universal surjectivity is defined as in \cite[p.\ 671]{Nakayama}.  That is, a morphism of fine and saturated logarithmic schemes $X \to Y$ is universally surjective if, for every fine and saturated logarithmic scheme $Y'$ and every morphism $Y' \to Y$, the base change $X' \to Y'$ is surjective on the underlying schemes.

	As an example, we may take $Y$ to be a point, with characteristic monoid $\mathbf N^2$.  For each homomorphism $\varphi : \mathbf N^2 \to \mathbf N$, let $X_\varphi$ be the universal logarithmic scheme over $Y$ such that the map $\mathbf N^2 \to \overnorm M_{X_\varphi}$ factors through $\varphi$.  Let $X$ be the disjoint union of all $X_\varphi$.

	I claim that the projection $X \to Y$ is universally surjective.  Indeed, if $Y' \to Y$ is a morphism of logarithmic schemes and $y'$ is a geometric point of $Y'$ then there must be at least one $\varphi : \mathbf N^2 \to \mathbf N$ that factors through $\overnorm M_Y \to \overnorm M_{Y',y'}$ (since $\overnorm M_{Y',y'}$ is finitely generated).  Therefore $y' \to Y$ factors through $X$, so $X \to Y$ is universally surjective.

	On the other hand, the Kato--Nakayama space of $Y$ is $S^1(1) \times S^1(1)$ while the Kato--Nakayama space of $X_\varphi$ is $S^1(1)$.  The map $X^{\rm KN}_\varphi \to Y^{\rm KN}$ winds around the two circles at commensurate rates.  In particular, if $a, b \in \mathbf R(1)$ are points whose ratio is irrational then the image of $(a,b) \in \mathbf R(1) \times \mathbf R(1)$ in $S^1(1) \times S^1(1) = Y^{\rm KN}$ does not lie in the image of $X_\varphi^{\rm KN}$ for any $\varphi$.

	If universal surjectivity is defined to require surjecitivity after base change to points with valuative logarithmic strutures (that do not necessarily have charts by finitely generated monoids) then the finite type hypothesis in Proposition~\ref{prop:surj-kn} can be suppressed and the proof can be simplified.
\end{remark}

Since taking the Kato--Nakayama space is left exact by definition, the proposition implies that the filtration~\eqref{eqn:3} induces a filtration~\eqref{eqn:14} on $M^{\KN}$:
\setcounter{equation}{\value{theorem}}
\begin{align} 
W_{-1} M^{\KN} & = T^{\KN} = \uHom(X^{\rm KN}, \logGm^{\KN}) = \uHom(X^{\rm KN}, S^1(1)) \notag \\
W_0 M^{\KN} & = G^{\KN} \label{eqn:14} \\
W_1 M^{\KN} & = M^{\KN} \notag
\end{align}
It has the following graded pieces:
\begin{align*}
	\gr_{-1} W_\bullet M^{\rm KN} = T^{\rm KN} \\
	\gr_0 W_\bullet M^{\rm KN} = A^{\rm KN} \\
	\gr_1 W_\bullet M^{\rm KN} = \mathrm B Y^{\rm KN}
\end{align*}
It also gives $A^{\KN}$ the structure of a variegated extension of $E^{\KN}$ by $F^{\KN}$ over $S^{\KN}$:
\begin{gather}
0 \to T^{\KN} \to G^{\KN} \to A^{\KN} \to 0 \tag{$F^{\KN}$} \\
	0 \to A^{\KN} \to [A^{\KN} / Y] \to \mathrm B Y^{\KN}\to 0 \tag{$E^{\KN}$}
\end{gather}
We note that these sequences are induced from sequences $E^{\rm an}$ and $F^{\rm an}$ on $S^{\rm an}$:
\begin{gather}
0 \to T^{\rm an} \to G^{\rm an} \to A^{\rm an} \to 0 \tag{$F^{\rm an}$} \\
	0 \to A^{\rm an} \to [A^{\rm an} / Y^{\rm an}] \to \mathrm B Y^{\rm an} \to 0 \tag{$E^{\rm an}$} 
\end{gather}
Here $T^{\rm an}= \uHom(X, \Gm^{\rm an}) = \uHom(X, \mathbf C^\ast)$ and $G^{\rm an}$ is the analytification of the semiabelian variety $G^{\rm alg}$ underlying $G$.

Let $\uExt^1_{S^{\KN}}(-, \mathbf Z)$ and $\uExt^1_{S^{\rm an}}(-, \mathbf Z)$ denote, respectively, the sheaves of extensions by $\mathbf Z$ over the topological spaces $S^{\KN}$ and $S^{\rm an}$.  We note that for an input with connected fibers over $S$, the sheaf $\uHom(-, \mathbf Z)$ is trivial, so that $\uExt^1(-, \mathbf Z)$ coincides with the presheaf of extensions.

\begin{lemma}
We have vanishings~\eqref{eqn:34} and~\eqref{eqn:35}:
\begin{gather} 
\uExt^2_{S^{\KN}}(A^{\KN}, \mathbf Z) = \uExt^2_{S^{\rm an}}(A^{\rm an}, \mathbf Z) = 0 \label{eqn:34} \\
\uExt^2_{S^{\KN}}(\mathrm B Y^{\KN}, \mathbf Z) = \uExt^2_{S^{\rm an}}(\mathrm B Y^{\rm an}, \mathbf Z) = 0 \label{eqn:35}
\end{gather}
\end{lemma}
\begin{proof}
In the case of~\eqref{eqn:34}, this is because $A$ is a family of abelian varieties, so $A^{\rm an}$ and $A^{\KN}$ are families of complex tori, over $S^{\rm an}$ and $S^{\KN}$, respectively.  In the case of~\eqref{eqn:35}, we identify $\uExt^2(\mathrm B Y, -)$ with $\uExt^1(Y, -)$, which vanishes because $Y$ is locally free.
\end{proof}

\begin{lemma} \label{lem:ext-Z}
We have canonical, compatible isomorphisms~\eqref{eqn:24} and~\eqref{eqn:25}:
\begin{gather}
\uExt^1_{S^{\rm an}}(\Gm^{\rm an}, \mathbf Z) = \mathbf Z(-1) \label{eqn:24} \\
\uExt^1_{S^{\KN}}(\logGm^{\KN}, \mathbf Z) = \mathbf Z(-1) \label{eqn:25}
\end{gather}
\end{lemma}
\begin{proof}
	Let $\mathcal O_{S^{\rm an}}$ denote the sheaf of holomorphic functions on $S^{\rm an}$ and let $\bar{\mathscr L}_{S^{\rm KN}}$ be the sheaf of continuous functions $S^{\rm KN} \to \mathbf R(1)$ such that the composition $S^{\rm KN} \to \mathbf R(1) \to S^1(1)$ is locally induced from a section of $M_S^{\rm gp}$.  We have exact sequences:
\begin{gather}
	0 \to \mathbf Z(1) \to \mathcal O_{S^{\rm an}} \to \Gm^{\rm an} \to 0 \label{eqn:26} \\
	0 \to \mathbf Z(1) \to \bar{\mathscr L}_{S^{\rm KN}} \to \logGm^{\KN} \to 0 \label{eqn:27}
\end{gather}
These induce long exact sequences:
	\begin{gather}
		\uHom_{S^{\rm an}}(\mathcal O_{S^{\rm an}}, \mathbf Z) \to \mathbf Z(-1) \to \uExt^1_{S^{\rm an}}(\Gm^{\rm an}, \mathbf Z) \to \uExt^1_{S^{\rm an}}(\mathcal O_{S^{\rm an}}, \mathbf Z) \\
		\uHom_{S^{\rm KN}}(\bar{\mathscr L}_{S^{\rm KN}}, \mathbf Z) \to \mathbf Z(-1) \to \uExt^1_{S^{\rm KN}}(\logGm^{\rm KN}, \mathbf Z) \to \uExt^1_{S^{\rm KN}}(\bar{\mathscr L}_{S^{\rm KN}}, \mathbf Z)
	\end{gather}
	Since $\mathcal O_{S^{\rm an}}$ and $\bar{\mathscr L}_{S^{\rm KN}}$ are sheaves of real vector spaces, they are divisible, so $\uHom_{S^{\rm an}}(\mathcal O_{S^{\rm an}}, \mathbf Z) = \uHom_{S^{\rm KN}}(\bar{\mathscr L}_{S^{\rm KN}}, \mathbf Z) = 0$.  

	The lemma will follow once we demonstrate that $\uExt^1_{S^{\rm an}}(\mathcal O_{S^{\rm an}}, \mathbf Z)$ and $\uExt^1_{S^{\rm KN}}(\bar{\mathscr L}_{S^{\rm KN}}, \mathbf Z)$ vanish.  Since the argument is the same for both, we give it only in the second case.  The underlying torsor of an extension of $\bar{\mathscr L}_{S^{\rm KN}}$ by $\mathbf Z$ is a $\mathbf Z$-torsor over $S^{\rm KN} \times \mathbf R(1)$.  Since this space has contractible fibers over $S^{\rm KN}$, this torsor must be locally trivial in $S^{\rm KN}$.  An extension is therefore described by maps $\bar{\mathscr L}_{S^{\rm KN}} \times \bar{\mathscr L}_{S^{\rm KN}} \to \mathbf Z$ encoding the group structure.  But these maps must be locally constant in $S^{\rm KN}$, again because $S^{\rm KN} \times \mathbf R(1)$ has contractible fibers over $S^{\rm KN}$, so the extension is trivial.
\end{proof}

Applying $\uExt^1(-, \mathbf Z)$ and the two lemmas, we obtain the following list of exact sequences, the former two inducing the latter two:
\begin{gather}
	0 \to \uExt^1_{S^{\rm an}}(A^{\rm an}, \mathbf Z) \to \uExt^1_{S^{\rm an}}(G^{\rm an}, \mathbf Z) \to X^{\rm an}(-1) \to 0 \tag{$\uExt^1(F^{\rm an}, \mathbf Z)$} \\
	0 \to \uHom_{S^{\rm an}}(Y^{\rm an}, \mathbf Z) \to \uExt^1_{S^{\rm an}}([A^{\rm an}/Y^{\rm an}], \mathbf Z) \to \uExt^1_{S^{\rm an}}(A^{\rm an}, \mathbf Z) \to 0 \tag{$\uExt^1(E^{\rm an}, \mathbf Z)$} \\
	0 \to \uExt^1_{S^{\rm KN}}(A^{\KN}, \mathbf Z) \to \uExt^1_{S^{\rm KN}}(G^{\KN}, \mathbf Z) \to X^{\KN}(-1) \to 0 \tag{$\uExt^1(F^{\KN}, \mathbf Z)$} \label{eqn:105} \\
	0 \to \uHom_{S^{\rm KN}}(Y^{\KN}, \mathbf Z) \to \uExt^1_{S^{\rm KN}}([A^{\KN}/ Y^{\KN}], \mathbf Z) \to \uExt^1_{S^{\rm KN}}(A^{\KN}, \mathbf Z) \to 0 \tag{$\uExt^1(E^{\KN}, \mathbf Z)$} \label{eqn:106}
\end{gather}

We also have a variegated extension $\uExt^1_{S^{\rm KN}}(M^{\KN}, \mathbf Z)$ of $\uExt^1_{S^{\rm KN}}(F^{\KN}, \mathbf Z)$ by $\uExt^1_{S^{\rm KN}}(E^{\KN}, \mathbf Z)$.  As in Section~\ref{sec:variegated}, this determines a class in~\eqref{eqn:17}:
\begin{equation} \label{eqn:17}
\uExt^1_{S^{\KN}}(X^{\KN}(-1), \uHom(Y^{\KN},\mathbf Z)) \Big/ \uExt^1_{S^{\rm an}}(X^{\rm an}(-1), \uHom(Y^{\rm an}, \mathbf Z))
\end{equation}
Since $X$, $Y$, and $\mathbf Z$ are all locally free sheaves, all such extensions are determined by the actions of the fundamental groups.  We may therefore identify the quotient with actions of the inertia group $I = \ker(\pi_1(S^{\KN}) \to \pi_1(S^{\rm an})) = \Hom(\overnorm M_S^{\rm gp}, \mathbf Z(1))$.  We therefore obtain a homomorphism~\eqref{eqn:18}:
\begin{equation} \label{eqn:18}
\Hom(\overnorm M_S^{\rm gp}, \mathbf Z(1)) \to \Hom(X(-1) \otimes Y, \mathbf Z)
\end{equation}
Equivalently, it is a pairing~\eqref{eqn:19}:
\begin{equation} \label{eqn:19}
X \otimes Y \to \overnorm M_S^{\rm gp}
\end{equation}

\begin{proposition} \label{prop:betti}
The pairing~\eqref{eqn:19} coincides with the logarithmic monodromy pairing.
\end{proposition}
\begin{proof}
We introduce the following commutative diagram:
\begin{equation*} \xymatrix{
0 \ar[d] & 0 \ar[d] \\
\uHom_S(X \otimes Y, \Gm)^{\rm an} \ar[r] \ar[d] & \uExt^1_{S^{\rm an}}(X \otimes Y, \mathbf Z(1)) \ar[d] \\
\uHom_S(X \otimes Y, \logGm)^{\rm an} \ar[r] \ar[d] & \rho_\ast \uExt^1_{S^{\KN}}(X \otimes Y, \mathbf Z(1)) \ar[d] \\
\uHom_S(X \otimes Y, \tropGm)^{\rm an} \ar@{-->}[r]^-\sim \ar[d] & \uHom_{S^{\rm an}}(I, \uHom_{S^{\rm an}}(X \otimes Y, \mathbf Z(1))) \ar[d] \\
0 & 0
} \end{equation*}
Here $\rho$ denotes the projection from $S^{\KN}$ to $S^{\rm an}$.

The sequence on the left is the one used to construct the monodromy pairing in Section~\ref{sec:monodromy}, and the one on the right is the one used to construct the pairing~\eqref{eqn:19}.  The horizontal maps are cup product with the extensions~\eqref{eqn:26} and~\eqref{eqn:27}.

Let $\partial : X \otimes Y \to \tropGm$ be the logarithmic monodromy pairing.  Choose a local lift, $\tilde\partial : X \otimes Y \to \logGm$.  This induces $\tilde\partial^{\KN} : X \otimes Y \to \logGm^{\KN} = S^1(1)$.  Now choose a local lift $\psi : X \otimes Y \to \mathbf R(1)$ of $\tilde\partial^{\KN}$.  Traversing a loop $\gamma : \overnorm M_S^{\rm gp} \to \mathbf Z(1)$ in $I$ replaces this lift with $\psi + \gamma \circ \partial$.  Therefore the monodromy homomorphism $I \to \Hom(X \otimes Y, \mathbf Z(1))$ sends $\gamma$ to $\gamma \circ \partial$, and the monodromy pairing of Section~\ref{sec:monodromy} coincides with~\eqref{eqn:19}.
\end{proof}

\section{The Hodge realization}
\label{sec:hodge}

We continue to work over $\mathbf C$.  Recall that the Kato--Nakayama space of a logarithmic scheme $S$ is equipped with a sheaf of rings, $\mathcal O_{S^{\rm KN}}$ \cite[\S3]{KN}.  In order to adapt the construction to logarithmic semiabelian varieties, we recall it.

Let $\rho : S^{\rm KN} \to S^{\rm an}$ denote the projection from the Kato--Nakayama space to the complex analytification.  By the definition of $S^{\rm KN}$, local sections of $\rho^{-1} M_{S^{\rm an}}^{\rm gp}$ determine continuous maps $S^{\rm KN} \to S^1(1)$.  Kato and Nakayama define $\mathscr L_{S^{\rm KN}}$ to be the sheaf of local sections of $\rho^{-1} M_{S^{\rm an}}^{\rm gp}$, together with a lift of the induced map $S^{\rm KN} \to S^1(1)$ to $\mathbf R(1)$; they define $\mathcal O_{S^{\rm KN}}$ to be the sheaf of commutative rings on $S^{\rm KN}$ freely generated by the extension of additive groups $\rho^{-1} \mathcal O_{S^{\rm an}} \to \mathscr L_{S^{\rm KN}}$.  See \cite[\S3]{KN} for further details.  We will write $\bar{\mathscr L}_{S^{\rm KN}}$ for the image of $\mathscr L_{S^{\rm KN}}$ in the sheaf of continuous maps from $S^{\rm KN}$ to $\mathbf R(1)$; in other words, $\bar{\mathscr L}_{S^{\rm KN}}$ is the sheaf of continuous functions $S \to \mathbf R(1)$ that can be represented locally by sections of $\rho^{-1} M_{S^{\rm an}}^{\rm gp}$.

We cannot use this construction directly on a logarithmic semiabelian variety $G$ because it has no underlying analytic space.  However, if $T = \uHom(X, \logGm)$ is a logarithmic torus over $S$, then $T^{\rm KN} = \uHom_{S^{\rm an}}(X^{\rm an}, S^1(1))$.  We may define $\rho^{-1} M_{T^{\rm an}}^{\rm gp}$ directly to be $\pi^{-1} \rho^{-1} (M_{S^{\rm an}}^{\rm gp} \times X^{\rm an})$, where $\pi : T^{\rm KN} \to S^{\rm KN}$ is the projection, and we have abusively written $X^{\rm an}$ for the sheaf of sections of $X^{\rm an}$ over $S^{\rm an}$.  The role of the sheaf of functions $\rho^{-1} \mathcal O_{T^{\rm an}}$ is played by $\pi^{-1} \rho^{-1} \mathcal O_{S^{\rm an}}$, and we use the construction of the last paragraph.  The same construction can be applied to a $T$-torsor by gluing, and therefore also to a logarithmic semiabelian variety (with constant degeneration).

When $Z^{\rm KN}$ and $W^{\rm KN}$ are Kato--Nakayama spaces, we define a morphism $Z^{\rm KN} \to W^{\rm KN}$ to be a continuous map $f$ on the underlying topological spaces and a commutative diagram:
\begin{equation*} \xymatrix{
		f^{-1} \rho^{-1} \mathcal O_{W^{\rm an}} \ar[r] \ar[d] & \rho^{-1} \mathcal O_{Z^{\rm an}} \ar[d] \\
		f^{-1} \mathscr L_{W^{\rm KN}} \ar[r] & \mathscr L_{Z^{\rm KN}}
	}
\end{equation*}
The upper horizontal arrow is required to be a ring homomorphism and the lower horizontal arrow is a group homomorphism.  If $f : Z^{\rm KN} \to W^{\rm KN}$ is an $S^{\rm KN}$-morphism of Kato--Nakayama spaces as above induces a morphism of the modules of de Rham differentials:
\begin{equation*}
	f^{-1} \Omega_{W^{\rm KN} / S^{\rm KN}} \to \Omega_{Z^{\rm KN} / S^{\rm KN}}
\end{equation*}

When considering homomorphisms and extensions below, we will need to work in a category $\mathscr C$ of Kato--Nakayama spaces over $S^{\rm KN}$ containing all ringed spaces locally isomorphic to the Kato--Nakayama spaces of logarithmic $1$-motifs.  Example~\ref{ex:smfs-torsor}, below, shows that the particular choice of $\mathscr C$ does not affect the homomorphisms or extensions when the domain is locally isomorphic to a logarithmic $1$-motif.  

\begin{example} \label{ex:smfs-torsor}
	We recall that extensions of a sheaf of abelian groups, $G$, by another sheaf of abelian groups, $H$, on a site can be specified entirely in terms of data on $G$, as follows:  First, an extension of $G$ by $H$ entails an $H$-torsor $Q$ on $G$.  The group structure of the extension is encoded by a map $\alpha : p_1^{-1} Q \times p_2^{-1} Q \to a^{-1} Q$ and a section $\epsilon$ of $e^{-1} Q$, where $p_1, p_2, a : G \times G \to G$ are the first projection, the second projection, and the addition map, respectively, and $e$ is the zero section of $G$.  The map $\alpha$ must be compatible with the action of $H \times H$ on $p_1^{-1} Q \times p_2^{-1} Q$ and of $H$ on $a^{-1} Q$ via the addition map $H \times H \to H$.  Finally, certain identities of morphisms of sheaves on $G$, on $G \times G$, and on $G \times G \times G$ must hold, expressing the unital, commutative, and associative nature of the group operation of $Q$ encoded by $\alpha$ and $\epsilon$.
\end{example}


We write $\undernorm\Omega^p$ for the sheaf on $\mathscr C$ whose value on $Z^{\rm KN}$ is $\Gamma(Z^{\rm KN}, \Omega^p_{Z^{\rm KN}/S^{\rm KN}})$ and $\undernorm{\mathcal O}^{\rm KN}$ for the sheaf taking the value $\Gamma(Z^{\rm KN}, \mathcal O_{Z^{\rm KN}})$.  It will also be convenient to write $\bar{\mathscr L}$ for the sheaf on $\mathscr C$ whose value on $Z^{\rm KN}$ is $\Gamma(Z^{\rm KN}, \bar{\mathscr L}_{Z^{\rm KN}})$.

\begin{lemma} \label{lem:de-rham}
	Let $S$ be a logarithmic scheme and let $\pi : G \to S$ be a logarithmic semiabelian variety over $S$.  Then the relative logarithmic de Rham complex $\Omega_{G^{\rm KN} / S^{\rm KN}}^\bullet$ is quasi\"isomorphic to $\pi^{-1} \mathcal O_{S^{\rm KN}}$.
\end{lemma}
\begin{proof}
	The logarithmic semiabelian variety $G$ is an extension of an abelian variety $A$ by a logarithmic torus $T = \uHom(X, \logGm)$.  Therefore $G^{\rm KN}$ is locally isomorphic to $T^{\rm KN} \mathop\times_{S^{\rm KN}} A^{\rm KN}$.  We have $\mathcal O_{T^{\rm KN} \mathop\times_{S^{\rm KN}} A^{\rm KN}} = \mathcal O_{A^{\rm an}} \mathop\otimes_{\mathcal O_{S^{\rm an}}} \mathcal O_{T^{\rm KN}}$ since the stalks of $\mathcal O_{T^{\rm KN}}$ are polynomial rings over the stalks of $\mathcal O_{S^{\rm an}}$.  Therefore $\Omega_{T^{\rm KN} \mathop\times_{S^{\rm KN}} A^{\rm KN} / S^{\rm KN}}^\bullet$ is isomorphic to the total complex of $\Omega_{T^{\rm KN} / S^{\rm KN}}^\bullet \otimes_{\mathcal O_{S^{\rm an}}} \Omega_{A^{\rm an} / S^{\rm an}}^\bullet  $.  It therefore suffices to observe that $\mathcal O_{S^{\rm an}} \to \Omega_{A^{\rm an} / S^{\rm an}}^\bullet$ and $\mathcal O_{S^{\rm KN}} \to \Omega_{T^{\rm KN} / S^{\rm KN}}^\bullet$ are quasi\"isomorphisms.  For $A^{\rm an}$ one may observe that the de Rham complex is locally pulled back from the holomorphic de Rham complex of $\mathbf C^k$ over a point; for $T^{\rm KN}$, the de Rham complex is locally pulled back from the algebraic de Rham complex of $\mathbf R^\ell$ over a point.
\end{proof}

\begin{lemma} \label{lem:log-torus-omega}
	Let $T = \uHom(X, \logGm)$ be a logarithmic torus over $S$ and let $\pi : T^{\rm KN} \to S^{\rm KN}$ be the projection.  Then we have the following identifications:
	\setcounter{equation}{\value{theorem}}
	\begin{gather}
		\pi_\ast \mathcal O_{T^{\rm KN}} = \mathcal O_{S^{\rm KN}} \label{eqn:74} \\
		\pi_\ast \Omega^1_{T^{\rm KN} / S^{\rm KN}} = \mathcal O_{S^{\rm KN}} \, d \log X \label{eqn:75} \\
		\pi_\ast \Omega^p_{T^{\rm KN} / S^{\rm KN}} = \bigwedge^p \pi_\ast \Omega^1_{T^{\rm KN}/S^{\rm KN}} \label{eqn:76}
	\end{gather}
\end{lemma}
\begin{proof}
	We may identify $\mathcal O_{T^{\rm KN}}$ with $\mathcal O_{S^{\rm KN}}[\log X]$.  This is a local system on the fibers of $T^{\rm KN}$ over $S^{\rm KN}$, with the fundamental group of the fibers, $\uHom(X,\mathbf Z(1))$, acting by $\gamma . \log x = \log x + \gamma(x)$.  A direct verification (for example, by induction on degree) shows that the only monodromy invariants in $\mathcal O_{S^{\rm KN}}[\log X]$ are the polynomials that are constant on the fibers, so we may conclude that $\pi_\ast \mathcal O_{T^{\rm KN}} = \mathcal O_{S^{\rm KN}}$, which is~\eqref{eqn:74}.

	We can identify $\Omega^1_{T^{\rm KN}} = \mathcal O_{S^{\rm KN}}[\log X]\, d \log X$.  Since the $d \log x$ for $x \in X$ are invariant under monodromy, the sheaf $\Omega^1_{T^{\rm KN}/S^{\rm KN}}$ is locally trivial in $S^{\rm KN}$, so this gives~\eqref{eqn:75}.  Finally, we have~\eqref{eqn:77}:
	\begin{equation} \label{eqn:77}
		\Omega^p_{T^{\rm KN} / S^{\rm KN}} = \mathcal O_{T^{\rm KN}} \otimes_{\pi^{-1} \mathcal O_{S^{\rm KN}}} \bigwedge^p \pi^{-1} \mathcal O_{S^{\rm KN}} \, d\log X
	\end{equation}
	We may therefore apply the projection formula to get~\eqref{eqn:76}.
\end{proof}

\setcounter{theorem}{\value{equation}}
\begin{lemma} \label{lem:hom-omega-p}
	Suppose that $M$ is a logarithmic $1$-motif over a logarithmic scheme $S$.  Then $\uHom_{S^{\rm KN}}(M^{\rm KN}, \undernorm\Omega^p) = 0$ for all $p \geq 2$.
\end{lemma}
\begin{proof}
	Fix $p \geq 2$ and present $M$ as $[G/Y]$ with $G$ an extension of an abelian variety $A$ by $T = \uHom(X, \logGm)$.  Since $G^{\rm KN}$ covers $M^{\rm KN}$, it is sufficient to show that $\uHom_{S^{\rm KN}}(G^{\rm KN}, \undernorm\Omega^p) = 0$, for which it suffices to show that $\uHom_{S^{\rm KN}}(A^{\rm KN}, \undernorm\Omega^p) = \uHom_{S^{\rm KN}}(T^{\rm KN}, \undernorm\Omega^p) = 0$.

	We show first that $\uHom_{S^{\rm KN}}(T^{\rm KN}, \undernorm\Omega^p) = 0$.  Every homomorphism $T^{\rm KN} \to \undernorm\Omega^p$ has an underlying section of $\Omega^p_{T^{\rm KN}/S^{\rm KN}}$.  But one may see by calculation in a basis that an element of $\bigwedge^p \mathcal O_{S^{\rm KN}} \, d \log X$ can only represent a homomorphism if it is zero or if $p = 1$.

	The proof for $A$ is similar.  Writing $\pi : A^{\rm KN} \to S^{\rm KN}$ for the projection, we may identify $\pi_\ast \Omega^1_{A^{\rm KN}}$ with $\mathcal O_{S^{\rm KN}} \mathop\otimes_{\mathcal O_S} V$, where $V$ is the contangent space of $A$ at the origin.  Then $\pi_\ast \Omega^p_{A^{\rm KN} / S^{\rm KN}} = \mathcal O_{S^{\rm KN}} \mathop\otimes_{\mathcal O_S} \bigwedge^p V$, and we see as before that a section of this group can only represent a homomorphism if it is zero or if $p = 1$.
\end{proof}

\begin{lemma} \label{lem:hom-O}
	Let $M$ be a logarithmic $1$-motif over a logarithmic scheme $S$.  Then we have $\uHom_{S^{\rm KN}}(M^{\rm KN}, \undernorm{\mathcal O}^{\rm KN}) = 0$.
\end{lemma}
\begin{proof}
	As in the last lemma, it is sufficient to prove the assertion when $M = T$ is a logarithmic torus over $S$ and when $M = A$ is an abelian variety over $S$.

	First we show that $\uHom_{S^{\rm KN}}(T^{\rm KN}, \undernorm{\mathcal O}^{\rm KN}) = 0$.  Let $\pi : T^{\rm KN} \to S^{\rm KN}$ be the projection.  By Lemma~\ref{lem:log-torus-omega}, we have $\pi_\ast \mathcal O_{T^{\rm KN}} = \mathcal O_{S^{\rm KN}}$.  A section of $\mathcal O_{S^{\rm KN}}$ can represent a homomorphism $T^{\rm KN} \to \undernorm{\mathcal O}^{\rm KN}$ only if it is zero.

	Next we show that $\uHom_{S^{\rm KN}}(A^{\rm KN}, \undernorm{\mathcal O}^{\rm KN}) = 0$.  Let $U$ be the tangent space at the origin of $A^{\rm an}$.  This is a complex vector bundle over $S^{\rm an}$ with strict logarithmic structure, so $U^{\rm KN}$ is a complex vector bundle over $S^{\rm KN}$.  Choosing local trivializations of $U^{\rm KN}$, homomorphisms $U^{\rm KN} \to \undernorm{\mathcal O}^{\rm KN}$ correspond to homogeneous linear functions with coefficients in $\mathcal O_S^{\rm KN}$.  These functions will descend to $A^{\rm KN}$ only if they are constant, hence $\uHom_{S^{\rm KN}}(A^{\rm KN}, \undernorm{\mathcal O}^{\rm KN}) = 0$.
\end{proof}

\begin{lemma} \label{lem:ext-omega}
	Let $M$ be a logarithmic $1$-motif over a logarithmic scheme $S$.  Then we have $\uExt^1_{S^{\rm KN}}(M^{\rm KN}, \undernorm{\Omega}^1) = 0$.
\end{lemma}
\begin{proof}
	Present $M$ as $[G/Y]$ where $G$ is a logarithmic semiabelian variety over $S$ and $Y$ is a lattice over $S$.  We have $\uHom_{S^{\rm KN}}(Y^{\rm KN}, \undernorm{\Omega}^1) = 0$ because $\undernorm\Omega^1$ vanishes on $Y$ (since $Y$ has discrete fibers and is strict over $S$), and $\uExt^1_{S^{\rm KN}}(Y^{\rm KN}, \undernorm\Omega^1) = 0$ since $Y$ is locally free.  The problem is therefore reduced to the case where $M = G$ is logarithmic semiabelian.  This reduces immediately to the cases where $M$ is a logarithmic torus and where $M$ is an abelian variety.  

	When $M = T = \uHom(X, \logGm)$ is a logarithmic torus with projection $\pi : T^{\rm KN} \to S^{\rm KN}$, we have $\mathcal O_{T^{\rm KN}} = \pi^{-1} \mathcal O_{S^{\rm KN}}[\log X]$.  Let $U = \uHom(X, \bar{\mathscr L})$ be the fiberwise universal cover of $T^{\rm KN}$ (recall that $\bar{\mathscr L}$ is the sheaf on $\mathscr C$ whose value on $Z^{\rm KN}$ is $\Gamma(Z, \bar{\mathscr L}_{Z^{\rm KN}})$).  Consider the exact sequence~\eqref{eqn:84}, with $X^\vee = \uHom(X,\mathbf Z(1))$ denoting the dual lattice of $X$:
	\setcounter{equation}{\value{theorem}}
	\begin{equation} \label{eqn:84}
		\uHom_{S^{\rm KN}}((X^\vee)^{\rm KN}, \undernorm\Omega^1) \to \uExt^1_{S^{\rm KN}}(T^{\rm KN}, \undernorm\Omega^1) \to \uExt^1_{S^{\rm KN}}(U, \undernorm\Omega^1)
	\end{equation}
	We have $\uHom_{S^{\rm KN}}((X^\vee)^{\rm KN}, \undernorm\Omega^1) = 0$ since $\Omega^1_{(X^\vee)^{\rm KN} / S^{\rm KN}}$ vanishes.  We argue that $\uExt^1_{S^{\rm KN}}(U, \undernorm\Omega^1) = 0$ as well.

	Since $U$ is homotopy equivalent over $S^{\rm KN}$ to $S^{\rm KN}$ and $\Omega^1_{U/S^{\rm KN}}$ is pulled back from $S^{\rm KN}$, the underlying torsor of any extension of $U$ by $\undernorm\Omega^1$ must be locally trivial in $S^{\rm KN}$.  Replacing $S^{\rm KN}$ by a cover, we now assume it is trivial.  Any extension is therefore determined by a morphism $\varphi : U \times_{S^{\rm KN}} U \to \undernorm\Omega^1$ specifying a group structure on $U \times_{S^{\rm KN}} \undernorm\Omega^1$.  This map must satisfy~\eqref{eqn:79} and~\eqref{eqn:80}, reflecting the associativity and identity properties of the group structure:
	\begin{gather}
		\varphi(u,v) + \varphi(u+v,w) = \varphi(u,v+w) + \varphi(v,w) \label{eqn:79} \\
		\varphi(0,u) = \varphi(u,0) = 0 \label{eqn:80}
	\end{gather}
	By a local choice of basis $X = \sum \mathbf Z x_i$, the $\log x_i$ become coordinates on $U$.  We write $\log x_i$ for the pullbacks of these coordinates to $U \times_{S^{\rm KN}} U$ via the first projection, and $\log y_i$ for their pullbacks via the second projection (and $\log z_i$ for their pullbacks via the third projection on $U \times_{S^{\rm KN}} U \times_{S^{\rm KN}} U$, in a moment).  We can then write $\varphi$ in the form~\eqref{eqn:81}, where $a_i$ and $b_i$ are polynomials in the $x_i$ and $y_i$:
	\begin{equation} \label{eqn:81}
		\varphi = \sum a_i \, d \log x_i + \sum b_i \, d \log y_i
	\end{equation}
	Substituting into~\eqref{eqn:79}, we obtain the equations~\eqref{eqn:82} (of differential forms on $U \mathop\times_{S^{\rm KN}} U \mathop\times_{S^{\rm KN}} U$):
	\begin{align} 
		a_i(u,v) \, d\log x_i + a_i(u+v,w) \, d\log x_i & = a_i(u,v+w)\, d\log x_i \notag \\
		b_i(u,v) \, d\log y_i + a_i(u+v,w) \, d\log y_i & = b_i(u,v+w)\, d\log y_i + a_i(v,w)\,d\log y_i \label{eqn:82} \\
		b_i(u+v,w)\, d\log z_i & = b_i(u,v+w)\, d\log z_i + b_i(v,w)\, d\log z_i \notag 
	\end{align}
	Substituting $u = 0$ into the first equation and applying~\eqref{eqn:80} shows that $a_i = 0$ for all $i$.  Likewise, substituting $w = 0$ into the last equation gives $b_i = 0$ for all $i$.  Thus all extensions of $U$ by $\undernorm\Omega^1$ are locally trivial in $S^{\rm KN}$, so $\uExt^1_{S^{\rm KN}}(T^{\rm KN}, \undernorm\Omega^1) = 0$.

	Now we show that $\uExt^1_{S^{\rm KN}}(A^{\rm KN}, \undernorm\Omega^1) = 0$ when $A$ is an abelian variety over $S$.  Let $\sigma : U \to A^{\rm KN}$ be the fiberwise universal cover.  We argue first that the restriction to $U$ of the underlying torsor of any extension of $A^{\rm KN}$ by $\undernorm\Omega^1$ must be locally trivial in $S^{\rm KN}$.  Indeed, the failure of local triviality is measured by a section of $\mathrm R^1 \pi_\ast \Omega^1_{A^{\rm KN} / S^{\rm KN}}$ over $S^{\rm KN}$.  We may identify $\Omega^1_{A^{\rm KN} / S^{\rm KN}} = \mathcal O_{S^{\rm KN}} \otimes_{\mathcal O_{S^{\rm an}}} \Omega^1_{A^{\rm an} / S^{\rm an}}$.  Since the stalks of $\mathcal O_{S^{\rm KN}}$ are polynomial rings over the stalks of $\mathcal O_{S^{\rm an}}$ \cite[Lemma~(3.3)]{KN}, they are in particular flat over $\mathcal O_{S^{\rm an}}$.  Using this observation and proper base change for locally compact topological spaces \cite[Proposition~II.2.6.7]{KS}, we may make the following identifications, with $\pi$ denoting the projections $A^{\rm KN} \to S^{\rm KN}$ and $A^{\rm an} \to S^{\rm an}$, and $\rho$ denoting $A^{\rm KN} \to A^{\rm an}$ and $S^{\rm KN} \to S^{\rm an}$:
	\begin{equation*}
		\mathrm R^1 \pi_\ast \Omega^1_{A^{\rm KN} / S^{\rm KN}} = \mathcal O_{S^{\rm KN}} \mathcal\otimes_{\mathcal O_{S^{\rm an}}} \mathrm R^1 \pi_\ast \rho^{-1} \Omega^1_{A^{\rm an} / S^{\rm an}} = \mathcal O_{S^{\rm KN}} \mathcal\otimes_{\mathcal O_{S^{\rm an}}} \rho^{-1} \mathrm R^1 \pi_\ast \Omega^1_{A^{\rm an} / S^{\rm an}}
	\end{equation*}
	Since the map $\rho^{-1} \mathrm R^1 \pi_\ast \Omega^1_{A^{\rm an} / S^{\rm an}} \to \mathrm R^1 (\pi \sigma)_\ast \Omega^1_{U/S^{\rm KN}}$ factors through $\rho^{-1} \mathrm R^1 (\pi\sigma)_\ast \Omega^1_{U^{\rm an} / S^{\rm an}} = 0$, it follows that the pullback of any $\Omega^1_{A^{\rm KN} / S^{\rm KN}}$ torsor to $U$ is locally trivial in $S^{\rm KN}$.
	
	We may now proceed by the same argument as with $T^{\rm KN}$.  The extension is determined by a map $\varphi : U \mathop\times_{S^{\rm KN}} U \to \undernorm\Omega^1$ satisfying~\eqref{eqn:79} and~\eqref{eqn:80}, except the $a_i$ and $b_i$ are now holomorphic functions in the $x_i$ and $y_i$.  The same argument shows that $\varphi = 0$, and therefore that $\uExt^1_{S^{\rm KN}}(A^{\rm KN}, \undernorm\Omega^1) = 0$, as required.
\end{proof}

Lemma~\ref{lem:hom-omega-p} supplies the first map in the sequence~\eqref{eqn:99} of Proposition~\ref{prop:hodge}:
\setcounter{theorem}{\value{equation}}
\begin{proposition} \label{prop:hodge}
	Let $M$ be a logarithmic $1$-motif over a logarithmic scheme $S$.  The sequence~\eqref{eqn:99} is exact:%
	\setcounter{equation}{\value{theorem}}
	\begin{equation} \label{eqn:99}
		0 \to \uHom_{S^{\rm KN}}(M^{\rm KN}, \undernorm{\Omega}^1) \to \uExt^1_{S^{\rm KN}}(M^{\rm KN}, \mathcal O_{S^{\rm KN}}) \to \uExt^1_{S^{\rm KN}}(M^{\rm KN}, \undernorm{\mathcal O}^{\rm KN}) \to 0
	\end{equation}
\end{proposition}
\begin{proof}
	We apply $\uHom_{S^{\rm KN}}(M^{\rm KN}, -)$ to the following sequence of sheaves on $\mathscr C$:
	\begin{equation} \label{eqn:85}
		0 \to \mathcal O_{S^{\rm KN}} \to \undernorm{\mathcal O}^{\rm KN} \to \undernorm\Omega^1 \to \undernorm\Omega^2 \to \cdots
	\end{equation}
	The sequence~\eqref{eqn:85} is not necessarily exact on all of $\mathscr C$, but Lemma~\ref{lem:de-rham} implies that it is exact when restricted to the subcategory of Kato--Nakayama spaces that are locally isomorphic to fiber products of copies of $M$ over $S$.  As explained in Example~\ref{ex:smfs-torsor}, the definitions of $\uHom_{S^{\rm KN}}(M^{\rm KN}, -)$ and $\uExt^1_{S^{\rm KN}}(M^{\rm KN}, -)$ only depend on the restriction of~\eqref{eqn:85} to this subcategory, so we obtain a spectral sequence converging to $0$, a piece of whose $E_1$ page is shown below:
	\begin{equation*} \vcenter{\xymatrix@C=8pt@R=10pt{
			0 \ar[r] & \uExt^1_{S^{\rm KN}}(M^{\rm KN}, \mathcal O_{S^{\rm KN}}) \ar[r] & \uExt^1_{S^{\rm KN}}(M^{\rm KN}, \undernorm{\mathcal O}^{\rm KN}) \ar[r] &  0 \ar[r] & \uExt^1_{S^{\rm KN}}(M^{\rm KN}, \undernorm\Omega^2) \\
			0 \ar[r] & \uHom_{S^{\rm KN}}(M^{\rm KN}, \mathcal O_{S^{\rm KN}}) \ar[r] & 0 \ar[r] & \uHom_{S^{\rm KN}}(M^{\rm KN}, \undernorm\Omega^1) \ar[r] & 0
	}} \end{equation*}
	We have used Lemmas~\ref{lem:hom-omega-p},~\ref{lem:hom-O}, and~\ref{lem:ext-omega} for the vanishings of the groups $\uHom_{S^{\rm KN}}(M^{\rm KN}, \undernorm\Omega^2)$, $\uHom_{S^{\rm KN}}(M^{\rm KN}, \undernorm{\mathcal O}^{\rm KN})$, and $\uExt^1_{S^{\rm KN}}(M^{\rm KN}, \undernorm\Omega^1)$. The exactness of the sequence~\eqref{eqn:99} now follows from the convergence of this spectral sequence to $0$ in the displayed region.
\end{proof}

Proposition~\ref{prop:hodge} gives a $2$-step filtration on $\uExt^1_{S^{\KN}}(M^{\KN}, \mathcal O_{S^{\rm KN}})$:
\begin{align*}
	F^0 & = \uExt^1_{S^{\KN}}(M^{\KN}, \mathcal O_{S^{\rm KN}}) \\
	F^1 & = \uHom_{S^{\KN}}(M^{\KN}, \undernorm{\Omega}^1)
\end{align*}
It has the following graded pieces:
\begin{align*}
	\gr_0 F & = \uExt^1_{S^{\KN}}(M^{\KN}, \undernorm{\mathcal O}^{\rm KN}) \\
	\gr_1 F & = \uHom_{S^{\KN}}(M^{\KN}, \undernorm{\Omega}^1) 
\end{align*}

We may simplify the description of $\gr_0 F$ using the following lemma:

\setcounter{theorem}{\value{equation}}
\begin{lemma} \label{lem:T}
	We have $\uHom_{S^{\rm KN}}(T^{\KN}, \undernorm{\mathcal O}^{\rm KN}) = \uExt^1_{S^{\rm KN}}(T^{\KN}, \undernorm{\mathcal O}^{\rm KN}) = 0$.  
\end{lemma}
\begin{proof}
	The vanishing of $\uHom_{S^{\rm KN}}(T^{\rm KN}, \undernorm{\mathcal O}^{\rm KN})$ is a special case of Lemma~\ref{lem:hom-O}.  Let $U = \uHom(X^{\rm KN}, \bar{\mathscr L})$ be the fiberwise universal cover of $T^{\rm KN}$.  Since we have $T^{\rm KN} = U / (X^\vee)^{\rm KN}$, where $X^\vee = \uHom(X, \mathbf Z(1))$, we have an exact sequence:
	\setcounter{equation}{\value{theorem}}
	\begin{equation} \label{eqn:89}
		\uHom_{S^{\rm KN}}(U, \undernorm{\mathcal O}^{\rm KN}) \to X(-1) \otimes \mathcal O_{S^{\rm KN}} \to \uExt^1_{S^{\rm KN}}(T^{\rm KN}, \undernorm{\mathcal O}^{\rm KN}) \to \uExt^1_{S^{\rm KN}}(U, \undernorm{\mathcal O}^{\rm KN})
	\end{equation}
	We can interpret $\uExt^1_{S^{\rm KN}}(U, \undernorm{\mathcal O}^{\rm KN})$ as the sheaf over $S^{\rm KN}$ of real algebraic extensions of the underlying real vector bundle of $U$ by $\mathbf R$.  All such extensions can be trivialized locally in $S^{\rm KN}$, so $\uExt^1_{S^{\rm KN}}(U, \undernorm{\mathcal O}^{\rm KN}) = 0$.%
	\footnote{We can also argue explicitly by cocycles.  First, to minimize notation, we reduce to the case where $X = \mathbf Z$ by working locally and using the additivity of $\uExt^1$, so $U$ is a trivial real vector bundle of rank~$1$ over $S^{\rm KN}$.  The underlying torsor of an extension in $\uExt^1_{S^{\rm KN}}(U, \undernorm{\mathcal O}^{\rm KN})$ is locally trivial in $S^{\rm KN}$, since $U$ has contractible fibers over $S$ and $\mathcal O_U$ is a constant sheaf on the fibers.  Such an extension is therefore specified by a polynomial function $\varphi$ on $U \times U$ with coefficients in $\mathcal O_{S^{\rm KN}}$ satisfying~\eqref{eqn:79} and~\eqref{eqn:80}.  That is $\varphi \in \mathcal O_{S^{\rm KN}}[u,v]$ where $u$ and $v$ represent the coordinate projections on $U \times U$.  Two such $\varphi$ are considered equivalent if they differ by $\psi(u+v) - \psi(u) - \psi(v)$ for a polynomial function $\psi$ on $U$.  Writing $\varphi(u,v) = \sum a_{i,j} u^i v^j$, expanding~\eqref{eqn:79}, and comparing coefficients gives~\eqref{eqn:88} for all $i,j,k > 0$:
	\begin{equation} \label{eqn:88}
		\binom{i+j}{i} a_{i+j,k} = \binom{j+k}{k} a_{i,j+k}
	\end{equation}
	That is, for each $n > 0$, there is a $b_n$ such that $a_{i,j} = \frac{1}{i!\,j!} b_n$ when $i + j = n$.  Set $\psi(u) = \sum \frac{1}{n!} b_n u^n$.  Then $\varphi$ is the coboundary of $\psi$, so $\varphi$ represents the trivial extension.}

	The first map in~\eqref{eqn:89} is inverse to the natural map $X(-1) \to \uHom(U, \undernorm{\mathcal O}^{\rm KN})$ coming from the identification $U = \uHom_{S^{\rm KN}}(X, \bar{\mathscr L})$.  In particular, it is surjective, so $\uExt^1_{S^{\rm KN}}(T^{\rm KN}, \undernorm{\mathcal O}^{\rm KN})$ vanishes.
\end{proof}

Lemma~\ref{lem:T} allows us to make the following observation:
\begin{equation*}
	\gr_0 F = \uExt^1_{S^{\KN}}([M^{\KN}/T^{\KN}], \undernorm{\mathcal O}^{\rm KN}) = \uExt^1_{S^{\KN}}([A^{\KN} / Y^{\KN}], \undernorm{\mathcal O}^{\rm KN})
\end{equation*}

In the following lemma, we write $\undernorm{\mathcal O}^{\rm an}$ for the sheaf on complex analytic spaces over $S$ whose value on $Z$ is $\Gamma(Z, \mathcal O_Z)$.

\setcounter{theorem}{\value{equation}}
\begin{lemma}
	Let $\rho : S^{\rm KN} \to S^{\rm an}$ be the projection from the Kato--Nakayama space to the complex analytification.  The natural map~\eqref{eqn:90} is an isomorphism:
	\begin{equation} \label{eqn:90}
		\mathcal O_{S^{\rm KN}} \mathop\otimes_{\rho^{-1} \mathcal O_{S^{\rm an}}} \rho^{-1} \uExt^1_{S^{\rm an}}([A^{\rm an} / Y^{\rm an}], \undernorm{\mathcal O}^{\rm an}) \to \uExt^1_{S^{\rm KN}}([A^{\rm KN} / Y^{\rm KN}], \undernorm{\mathcal O}^{\rm KN})
	\end{equation}
\end{lemma}
\begin{proof}
	We note that any extension in $\uExt^1_{S^{\rm KN}}([A^{\rm KN}  / Y^{\rm KN}], \undernorm{\mathcal O}^{\rm KN})$ is uniquely determined by the family of torsors obtained by restricting it to $[A^{\rm KN} / Y^{\rm KN}]$ and its fiber products over $S^{\rm KN}$.  
	Since the logarithmic analytic stack $[A^{\rm an} / Y^{\rm an}]$ and all of its fiber products are strict over $S^{\rm an}$, the sheaf $\undernorm{\mathcal O}^{\rm KN}$ agrees with $\mathcal O_{S^{\rm KN}} \otimes_{\mathcal O_S} \undernorm{\mathcal O}^{\rm an}$ on these spaces.  
	Since the stalks of $\mathcal O_{S^{\rm KN}}$ are free over the stalks of $\mathcal O_{S^{\rm an}}$, we therefore have~\eqref{eqn:91}:
	\begin{equation} \label{eqn:91}
		\uExt^1_{S^{\rm KN}}([A^{\rm KN} / Y^{\rm KN}], \undernorm{\mathcal O}^{\rm KN}) = \uExt^1_{S^{\rm KN}}([A^{\rm KN} / Y^{\rm KN}], \rho^{-1} \undernorm{\mathcal O}^{\rm an}) \otimes_{\rho^{-1} \mathcal O_{S^{\rm an}}} \mathcal O_{S^{\rm KN}}
	\end{equation}
	It remains to show that the natural map~\eqref{eqn:92} is an isomorphism:
	\begin{equation} \label{eqn:92}
		\rho^{-1} \uExt^1_{S^{\rm an}}([A^{\rm an} / Y^{\rm an}], \undernorm{\mathcal O}^{\rm an}) \to \uExt^1_{S^{\rm KN}}(\rho^{-1} [A^{\rm an} / Y^{\rm an}], \rho^{-1} \undernorm{\mathcal O}^{\rm an})
	\end{equation}
	Recall that an extension of $[ A^{\rm an} / Y^{\rm an} ]$ by $\undernorm{\mathcal O}^{\rm an}$ consists of an extension of $A^{\rm an}$ by $\undernorm{\mathcal O}^{\rm an}$ and a trivialization of the induced extension of $Y^{\rm an}$ by $\undernorm{\mathcal O}^{\rm an}$.  The two sides of the arrow~\eqref{eqn:92} are thus torsors over the two sides of~\eqref{eqn:93}, under the group $\uHom_{S^{\rm KN}}(Y^{\rm KN}, \rho^{-1} \undernorm{\mathcal O}^{\rm an})$:
	\begin{equation} \label{eqn:93}
		\rho^{-1} \uExt^1_{S^{\rm an}}(A^{\rm an}, \undernorm{\mathcal O}^{\rm an}) \to \uExt^1_{S^{\rm KN}}(\rho^{-1} A^{\rm an}, \rho^{-1} \undernorm{\mathcal O}^{\rm an})
	\end{equation}
	The morphism respects the action of $\uHom_{S^{\rm KN}}(Y^{\rm KN}, \rho^{-1} \undernorm{\mathcal O}^{\rm an})$, so it suffices to demonstrate that~\eqref{eqn:93} is an isomorphism.  

	Recall from Example~\ref{ex:smfs-torsor} that an extension of $A^{\rm an}$ by $\undernorm{\mathcal O}^{\rm an}$ is specified by a $\mathcal O_{A^{\rm an}}$-torsor on $A^{\rm an}$, a morphism of $\mathcal O_{A^{\rm an} \mathop\times_{S^{\rm an}} A^{\rm an}}$-torsors on $A^{\rm an} \mathop\times_{S^{\rm an}} A^{\rm an}$, a morphism of $\mathcal O_{S^{\rm an}}$-torsors on $S^{\rm an}$, and various compatibilities among these, expressed as morphisms of $\mathcal O_{(-)^{\rm an}}$-torsors on $A^{\rm an}$, $A^{\rm an} \mathop\times_{S^{\rm an}} A^{\rm an}$, and $A^{\rm an} \mathop\times_{S^{\rm an}} A^{\rm an} \mathop\times_{S^{\rm an}} A^{\rm an}$.  Since $A^{\rm an}$ and all of its fiber products are proper over $S^{\rm an}$, all of these data commute with base change to $S^{\rm KN}$ by proper base change for cohomology on locally compact, Hausdorff topological spaces \cite[Proposition~II.2.6.7]{KS}.  We conclude that~\eqref{eqn:93} is an isomorphism, as required.
\end{proof}

By the lemma, we conclude:
\begin{equation*}
	\gr_0 F = \uExt^1_{S^{\rm an}}([A^{\rm an} / Y^{\rm an}], \undernorm{\mathcal O}^{\rm an}) \otimes_{\mathcal O_{S^{\rm an}}} \mathcal O_{S^{\rm KN}}
\end{equation*}

Next we compute the induced filtration $F^i \uExt^1_{S^{\rm KN}}(\gr_j W^{\rm KN}, \mathcal O_{S^{\rm KN}})$, for each of the pieces $\gr_j W^{\rm KN}$ of the logarithmic monodromy filtration, $W^{\rm KN}$.  We have $\gr_{-1} W^{\rm KN} = T^{\rm KN} = \uHom(X, S^1(1))$.  By Lemma~\ref{lem:T}, $\gr_0 F\, \uExt^1_{S^{\rm KN}}(\gr_{-1} W^{\rm KN}, \mathcal O_{S^{\rm KN}}) = \uExt^1_{S^{\rm KN}}(T^{\rm KN}, \undernorm{\mathcal O}^{\rm KN}) = 0$.  Therefore we have the following equalities:
\begin{equation*}
	\uExt^1_{S^{\rm KN}}(\gr_{-1} W^{\rm KN}, \mathcal O_{S^{\rm KN}}) = F^1 \uExt^1_{S^{\rm KN}}(\gr_{-1} W^{\rm KN}, \mathcal O_{S^{\rm KN}}) = \uHom_{S^{\rm KN}}(T^{\rm KN}, \undernorm{\Omega}^1) = X(-1) \otimes \mathcal O_{S^{\rm KN}}
\end{equation*}
The final equality comes from $\Gamma(T^{\rm KN}, \Omega^1_{T^{\rm KN}/ S^{\rm KN}}) = \mathcal O_{S^{\rm KN}} \, d \log X$, by Lemma~\ref{lem:log-torus-omega}.  The choice of twist comes from $\uExt^1_{S^{\rm KN}}(T^{\rm KN}, \mathbf Z) = X(-1)$ (Lemma~\ref{lem:ext-Z}).

On the middle piece $\gr_0 W^{\KN} = A^{\KN}$, the filtration is the usual Hodge filtration on the first cohomology of the complex abelian variety $A^{\rm an}$.

On the last piece, $\gr_1 W^{\KN} = \mathrm B Y^{\KN}$, we have $\uHom_{S^{\KN}}(\mathrm B Y^{\KN}, \undernorm \Omega^1) = 0$ and we calculate:
\begin{equation*}
	\uExt^1_{S^{\KN}}(\mathrm B Y^{\KN}, \undernorm{\mathcal O}^{\rm KN}) = \uHom_{S^{\KN}}(Y^{\KN}, \undernorm{\mathcal O}^{\rm KN}) = \uHom_{S^{\KN}}(Y^{\KN}, \mathcal O_{S^{\rm KN}})
\end{equation*}

We summarize these calculations in the following table:
\def\arraystretch{1.5}
\begin{equation} \label{eqn:97}
\begin{array}{r|ccc} 
	\gr_i F\, \uExt^1(\gr_j W^{\rm KN}, \mathcal O_{S^{\rm KN}}) & \gr_0 F & \gr_1 F \\
\hline
	\gr_1 W^{\rm KN}  & \uHom_{S^{\rm KN}}(Y^{\rm KN},\mathcal O_{S^{\rm KN}}) & 0  \\
	\gr_0 W^{\rm KN}  & \mathcal H^1(A^{\rm an}, \mathcal O_A) \otimes_{\mathcal O_{S^{\rm an}}} \mathcal O_{S^{\rm KN}} & \mathcal H^0(A^{\rm an}, \Omega^1_{A^{\rm an}}) \otimes_{\mathcal O_{S^{\rm an}}} \mathcal O_{S^{\rm KN}} \\
	\gr_{-1} W^{\rm KN} & 0 &  X(-1) \otimes \mathcal O_{S^{\rm KN}}
\end{array}
\end{equation}

\setcounter{subsection}{\value{equation}}
\subsection{Limiting mixed Hodge structure} \label{sec:lmhs}

Recall that if $V$ is a finite dimensional vector space and $N$ is a nilpotent endomorphism of $V$ then there is a unique filtration $W$ on $V$ such that $N W_i \subset W_{i-2}$ for all $i$ and $N^i$ induces an isomorphism $\gr_i W \to \gr_{-i} W$ for all $i$ (see \cite[(1.5.5)]{Ill94} or \cite[Proposition~(2.1)]{SZ85}).  When $N^2 = 0$, this filtration has the following simple form:
\begin{equation*}
	W_{-1} = \image N \qquad W_0 = \ker N \qquad W_1 = V
\end{equation*}
When $N$ is the logarithm of the monodromy acting unipotently on the cohomology of a smooth family of complex varieties over the punctured disc, this filtration is called the monodromy filtration.

Returning to the calculation in~\eqref{eqn:97}, we specialize to the case where $M$ is the logarithmic Jacobian of a logarithmic curve $C$ obtained by a $1$-parameter degeneration.  In this case, we have $X = Y = H = H_1(\mathfrak C)$, where $\mathfrak C$ is the dual graph of $C$.  By Proposition~\ref{prop:betti}, the logarithm of the monodromy action is given by the tropical intersection pairing, $H \otimes H \to \overnorm M_S = \mathbf Z$.  This is positive definite, and in particular is nondegenerate.  Therefore, the logarithm~\eqref{eqn:98} of the monodromy homomorphism on the rational cohomology is an isomorphism.
\setcounter{equation}{\value{subsection}}
\begin{equation} \label{eqn:98}
	N : \gr_1 W_{\mathbf Q} = H(-1) \otimes \mathbf Q \to \Hom(H,\mathbf Q) = \gr_{-1} W_{\mathbf Q}
\end{equation}
Thus $W_\bullet$ is the monodromy filtration on $\uExt^1(M^{\rm KN}, \mathbf Z)$.

\section{Picard--Lefschetz theory}
\label{sec:picard-lefschetz}

In this section, we consider a smooth curve $C$ over $T$, where $T$ is a smooth, non-complete algebraic curve.  Let $\overnorm T$ be the regular completion of $T$ and we assume that $\overnorm C$ is a semistable extension of $C$ to $\overnorm T$.  We give $\overnorm T$ the divisorial logarithmic structure.  Let $S$ be a point of $\overnorm T - T$, with the logarithmic structure induced from $\overnorm T$.

In this case, the proof of~\cite[Theorem~4.15.7]{logpic} shows that the logarithmic Jacobian of $\overnorm C_S$ over $S$ is constructed with $X = Y = H$, where $H = H_1(\mathfrak C)$, and $\mathfrak C$ is the dual graph of the central fiber.  The pairing $\partial : H \otimes H \to \overnorm M_S^{\rm gp} = \mathbf Z$ is the tropical intersection pairing (see Section~\ref{sec:monodromy}).  As usual, we also regard $\partial$ as a homomorphism $H \to \Hom(H, \mathbf Z)$ and as a homomorphism $H_n \to \Hom(H_n,\mathbf Z/n\mathbf Z)$.

We wish to describe the action of a loop in $T$ around $S$ on $H^1(C)$, for various cohomology theories, in terms of the monodromy pairing.  First we consider $H^1_{\et}(C, \mu_n)$, where $n$ is invertible in $\mathcal O_{\overnorm T}$.  We may identify $H^1_{\et}(C, \mu_n)$ with $J[n]$, where $J$ is the Jacobian of $C$ over~$T$.

Let $\overnorm J$ denote the logarithmic Jacobian of $\overnorm C$ over $\overnorm T$, and $\overnorm J_S$ its restriction to $S$.  Since $H^1_{\loget}(C, \mu_n) = \overnorm J[n]$ is a locally constant sheaf on the logarithmic \'etale site of $\overnorm T$, it will suffice to describe the monodromy action over $S$.

The exact sequence \ref{eqn:103} from Section~\ref{sec:torsion} gives a surjective map $\overnorm J_S[n] \to H/nH$ that we will denote $\alpha \mapsto \overnorm\alpha$.  The exact sequence \ref{eqn:104} gives an inclusion $\Hom(H,\mu_n) \subset \overnorm J_S[n]$.

\begin{corollary} \label{cor:PL1}
Let $g$ be the image of $\gamma \in \pi_1^{\et}(T)$ under the projection $\pi_1^{\et}(T) \to \mu_n$ giving the monodromy action on $t^{1/n}$, where $t$ is a local parameter for $T$ at $S$.  The monodromy action of $\gamma$ on $H^1_{\et}(C, \mu_n)$ is given by the following formula for $\alpha \in H^1_{\et}(C, \mu_n)$:
\begin{equation*}
\alpha \mapsto \alpha + g \partial(\overnorm\alpha)
\end{equation*}
\end{corollary}
\begin{proof}
	It was shown in Section~\ref{sec:etale} that the logarithmic inertia group operates on $\overline J[n]$ via the monodromy pairing.
\end{proof}

Next, we consider the same question when $\overnorm T$ is an analytic disc.  In this case, the universal coefficients theorem implies that pullback along the Abel map induces an isomorphism:
\begin{equation*}
\Ext^1(J, \mathbf Z) \to H^1(C, \mathbf Z)
\end{equation*}
Therefore the monodromy action can be described by the action of $\pi_1(T)$ on $\Ext^1(J, \mathbf Z)$.  Since $H^1(\overnorm C^{\KN}, \mathbf Z)$ and $\Ext^1(\overnorm J^{\KN}, \mathbf Z)$ are locally constant over $\overnorm T^{\KN}$, we may compute the monodromy action over $S^{\KN}$.  

The exact sequences \ref{eqn:105} and \ref{eqn:106} from Section~\ref{sec:betti} induce homomorphisms:
\begin{gather*}
\Ext^1(\overnorm J_S^{\KN}, \mathbf Z) \to H(-1) : \alpha \mapsto \overnorm\alpha \\
\Hom(H,\mathbf Z) \to \Ext^1(\overnorm J_S^{\KN}, \mathbf Z)
\end{gather*}
The monodromy pairing gives a homomorphism $\partial : H \to \Hom(H,\mathbf Z)$.

\begin{corollary} \label{cor:PL2}
Let $g \in \mathbf Z(1)$ give the monodromy action of $\gamma \in \pi_1(T)$ on $\log(t)$, where $t$ is a local parameter for $T$ at $S$.  The monodromy action of $\gamma$ on $H^1(C, \mathbf Z)$ is given by the following formula:
\begin{equation*}
\alpha \mapsto \alpha + g \partial(\overnorm\alpha)
\end{equation*}
\end{corollary}
\begin{proof}
	It was shown in Section~\ref{sec:betti} that $\pi_1(S^{\rm KN})$ acts on $\Ext^1(\overline J_S^{\rm KN}, \mathbf Z)$ via the monodromy pairing.
\end{proof}

\begin{remark}
The formulas in Corollaries~\ref{cor:PL1} and~\ref{cor:PL2} are given in terms of the tropical edge length pairing, not the intersection pairing on the cohomology of the curve.  These are related, but we require the principal polarization of the logarithmic Jacobian to describe the relationship.  This will be explained in another paper.
\end{remark}

\section*{Acknowledgements}

I am very grateful to Luc Illusie, who asked the question that inspired this work, and to my collaborators Samouil Molcho and Martin Ulirsch.  They elected not to be coauthors of this paper, but conversations with them were nevertheless essential to my understanding of these topics (my misunderstandings remain my own).  David Holmes provided valuable feedback on an earlier draft of this paper.  I also wish to express my deepest gratitude to the anonymous referee, whose patient and careful reading helped me to correct many mistakes, both large and small.

I am happy to acknowledge the hospitality of Brown University, where part of this paper was written, and the support of the Simons Foundation, awards 636210 and 822534.

\appendix
\section{The logarithmic and tropical multiplicative groups}
\label{sec:log-trop-Gm}

We work here in the category of fine and saturated logarithmic schemes.  The logarithmic multiplicative group is the functor $\logGm(S) = \Gamma(S, M_S^{\rm gp})$ and the tropical multiplicative group is $\tropGm(S) = \Gamma(S, \overnorm M_S^{\rm gp})$, where $\overnorm M_S^{\rm gp} = M_S^{\rm gp} / \mathcal O_S^\ast$ is the associated group of the characteristic monoid.  By definition $\tropGm = \logGm/\Gm$.  A \emph{logarithmic} (resp.\ \emph{tropical}) \emph{torus} over a logarithmic scheme $S$ is a sheaf on the big strict \'etale site of logarithmic schemes over $S$ that is strict-\'etale locally isomorphic to a finite product of copies of $\logGm$ (resp.\ $\tropGm$).  A \emph{lattice} over a logarithmic scheme $S$ is a sheaf on the big strict \'etale site of logarithmic schemes over $S$ that is locally isomorphic to a constant sheaf of finitely generated free abelian groups.  It follows from Propositions~\ref{prop:log-hom-ext} (resp.~\ref{prop:tropical-hom-ext}), below, that every logarithmic (resp.\ tropical) torus over a logarithmic scheme $S$ is isomorphic to $\uHom(X, \logGm)$ (resp.\ $\uHom(X, \tropGm)$) for some lattice $X$ over $S$.

We collect some basic facts about homomorphisms and extensions involving $\logGm$ and $\tropGm$, analogous to familiar properties of the conventional multiplicative group, $\Gm$.

\begin{proposition} \label{prop:tropGm-map}
	All morphisms $(\tropGm)_S \to (\tropGm)_S$ over a logarithmic scheme $S$ have the form $\alpha \mapsto \lambda + n \alpha$ for a unique $\lambda \in \Gamma(S, \overnorm M_S^{\rm gp})$ and a unique locally constant integer-valued function $n$ on $S$.  Such a morphism is a homomorphism if and only if $\lambda = 0$.
\end{proposition}
\begin{proof}
	Let $\varphi : (\tropGm)_S \to (\tropGm)_S$ be an $S$-morphism.  Give $\mathbf A^1$ the toric logarithmic structure.  Composing $\varphi$ with the projection $\pi : \mathbf A^1_S \to (\tropGm)_S$ gives a section of $\overnorm M_{\mathbf A^1_S}^{\rm gp}$.  We may identify $\Gamma(\mathbf A^1, \overnorm M_{\mathbf A^1_S}^{\rm gp}) = \Gamma(S, \overnorm M_S^{\rm gp}) \times \Hom(S, \mathbf Z)$.  Let $\lambda$ and $n$ be the two components of $\pi^\ast(\varphi)$ under this identification.

	Now, suppose that $T$ is any logarithmic scheme over $S$ and $\alpha \in \Gamma(T, \overnorm M_T)$ is a section of the characteristic monoid of $T$.  Then there is, locally in $T$, a morphism $T \to \mathbf A^1$ such that $\alpha$ is pulled back from the generator $1$ of $\Gamma(\mathbf A^1, \overnorm M_{\mathbf A^1}) \simeq \mathbf N$.  Therefore $\varphi(\alpha)$ is the pullback of $(\lambda, n)$, namely $\lambda + n \alpha$.

	The same reasoning works if $-\alpha \in \Gamma(T, \overnorm M_T)$.  In general, neither $\alpha$ nor $-\alpha$ lies in $\Gamma(T, \overnorm M_T)$, but we can still find a logarithmic modification $p : T' \to T$ such that, locally in $T'$, either $p^\ast \alpha$ or $- p^\ast \alpha$ does lie in $\overnorm M_{T'}$.  Therefore we have $\varphi_{T'}(p^\ast \alpha) = \lambda + n p^\ast(\alpha)$.  Since $\Gamma(T, \overnorm M_T^{\rm gp}) \to \Gamma(T', \overnorm M_{T'}^{\rm gp})$ is injective, we may conclude.

	Finally, $n$ and $\lambda$ are unique because they were obtained by evaluating the map $\varphi$ on~$\mathbf A^1_S$.
\end{proof}

\begin{proposition} \label{prop:tropical-hom-ext}
	Suppose that $T = \uHom(X,\tropGm)$ and $T' = \uHom(X', \tropGm)$ are tropical tori over a logarithmic scheme $S$.  Then $\uHom(T,T') = \uHom(X',X)$ and $\uExt^1(T,T') = 0$.
\end{proposition}
\begin{proof}
	These assertions are local in the \'etale topology on $S$, so we may assume that $X$ and $X'$ are free.  We may therefore reduce to the case where $X = X' = \mathbf Z$, so $T = T' = \tropGm$.  The first claim is then the conclusion of Proposition~\ref{prop:tropGm-map}.

	Now we consider an extension $W$ of $\tropGm$ by itself.  We argue first that the underlying torsor of such an extension can be trivialized locally in $S$.

	Give $\mathbf P^1$ its toric logarithmic structure and let $\pi : \mathbf P^1_S \to S$ be the base change to $S$.  We have $\Gamma(\mathbf P^1, \overnorm M_{\mathbf P^1}) = \mathbf N^2$, with the two generators corresponding to local parameters at $0$ and $\infty$.  Let $\alpha : \mathbf P^1_S \to \tropGm$ be the composition of the projection $\mathbf P^1_S \to \mathbf P^1$ with the morphism $\mathbf P^1 \to \tropGm$ corresponding to the section $(1,-1) \in \mathbf Z^2 \simeq \Gamma(\mathbf P^1, \overnorm M_{\mathbf P^1}^{\rm gp})$.  

	The $\tropGm$-torsor $\alpha^\ast W$ over $\mathbf P^1_S$ restricts to a $\overnorm M_{\mathbf P^1_S}^{\rm gp}$-torsor on the small \'etale site of $\mathbf P^1_S$.  We claim that $\mathrm R^1 \pi_\ast \overnorm M_{\mathbf P^1_S}^{\rm gp} = 0$.  By proper base change for \'etale cohomology, it is sufficient to prove this on the geometric fibers.  But on each fiber, $\overnorm M_{\mathbf P^1_S}^{\rm gp}$ is an extension of a sheaf that is concentrated in relative dimension~$0$, so has vanishing $H^1$, by a constant sheaf, which has vanishing $H^1$ because each geometric fiber is $\mathbf P^1$.

	We may therefore select a trivialization of $\alpha^\ast W$ over $\mathbf P^1_S$.  The restriction of $\alpha^\ast W$ to the dense torus of $\mathbf P^1_S$ is canonically trivialized (since $\alpha : \mathbf P^1_S \to \tropGm$ restricts to $0$ there).  Since every section of $\overnorm M_{\mathbf P^1_S}^{\rm gp}$ over the dense torus lifts to a global section, we may choose the trivialization of $\alpha^\ast W$ so that it restricts to the canonical trivialization on the dense torus.

	Any choice of trivialization is equivariant with respect to the action of $\Gm$ on $\mathbf P^1_S$, so this trivialization descends to $[ \mathbf P^1 / \Gm ]$.  In what remains of the proof, we will only make use of the restriction of this trivialization to $[ \mathbf A^1 / \Gm ] \subset [ \mathbf P^1 / \Gm ]$.  We abbreviate $[ \mathbf A^1 / \Gm ]$ to $\mathscr A$.  We will now propagate this trivialization to all logarithmic schemes over $(\tropGm)_S$.%
	\footnote{The trivialization we eventually produce will give a new trivialization of $\alpha^\ast W$ over $\mathbf P^1_S$, but this is not necessarily the same one that was chosen above.}

	Suppose that $Z$ is a logarithmic scheme over $S$ and $\beta \in \Gamma(Z, \overnorm M_Z^{\rm gp})$ describes a morphism $\beta : Z \to \tropGm$.  Let $Z_0 \subset Z$ be the open subset where $\beta = 0$ and let $Z_1$ be its closed complement (with the reduced scheme structure, say).  Let $U_+$ be the universal logarithmic scheme over $Z$ where $\beta \in \overnorm M_{U_+}$ (note that $U_+$ is not a subscheme of $Z$ but rather an open subscheme of the logarithmic blowup of $Z$ at the fractional ideal generated by $0$ and $\beta$) and let $Z_+ \subset U_+$ be the closed subscheme where the map $\mathcal O_{U_+}(-\beta) \to \mathcal O_{U_+}$ vanishes (here $\mathcal O_{U_+}(-\beta)$ is the invertible sheaf associated with the fiber of $M_{U_+}$ over $\beta$ in $\overnorm M_{U_+}$).    Similarly, let $U_-$ be the universal logarithmic scheme over $Z$ where $-\beta \in \overnorm M_{U_-}$ and let $Z_- \subset U_-$ be the closed scheme where $\mathcal O_{U_-}(\beta) \to \mathcal O_{U_-}$ vanishes.  

	We argue that $Z_+ \to Z_1$ and $Z_- \to Z_1$ are both, topologically, the inclusions of closed subsets and that $Z_+ \cup Z_- = Z_1$.  Locally in $Z$, we can find a map $Z \to \mathbf A^2$ (where $\mathbf A^2$ has the toric logarithmic structure) such that $\beta$ is the pullback of the global section $(1,-1)$ of $\overnorm M_{\mathbf A^2}^{\rm gp}$.  Then $Z_+$ and $Z_-$ are the fine and saturated pullbacks of the strict transforms of the axes in the blowup of $\mathbf A^2$ at the origin.  Since the fine pullback preserves closed embeddings, and saturation only adds nilpotent elements to the structure sheaf in this case, the maps $Z_+ \to Z$ and $Z_- \to Z$ are topologically the inclusions of closed subsets.  To see that $Z_+$ and $Z_-$ cover $Z_1$, suppose that $z$ is a valuative geometric point of $Z_1$.  Then either $\beta \big|_z > 0$ or $\beta\big|_z = 0$ or $\beta \big|_z < 0$ in the partial order on $\overnorm M_z^{\rm gp}$ induced by $\overnorm M_z$.  If $\beta \big|_z > 0$ then $z$ lies in the image of $Z_+$ and if $\beta \big|_z < 0$ then $z$ lies in the image of $Z_-$.  If $\beta \big|_z = 0$ then the image of $z$ under the map $Z \to \mathbf A^2$ described above must be the origin (since $z$ is not in $Z_0$, by assumption).  Thus the fiber over $z$ of the blowup of the origin in $\mathbf A^2$ meets both $Z_+$ and $Z_-$.  We conclude that the images of $Z_+$ and $Z_-$ are closed subsets that cover $Z_1$.

	We will describe a trivialization $W_\beta = \beta^\ast W$ over $Z$ by giving compatible trivializations of $W_\beta$ over $Z_0$, $Z_+$, and $Z_-$.%
	\footnote{To give a trivialization of $W_\beta$ over $Z$ is equivalent to giving a section of the induced sheaf on the small \'etale site of $Z$.  Let $i : Z_1 \to Z$ and $j : Z_0 \to Z$ denote the inclusions of the complementary closed and open sets $Z_1$ and $Z_0$.  Recall \cite[Th\'eor\`eme~9.5.4]{sga4-IV} that to give a section of $W_\beta$ over $Z$, it is equivalent to give a section $a$ of $i^\ast W_\beta$ on $Z_1$ and $b$ of $j^\ast W_\beta$ such that the image of $a$ under the natural homomoprhism $i^\ast W_\beta \to i^\ast j_\ast j^\ast W_\beta$ is $b$.

	For the same reason, to give a section of $W_\beta$ over $Z_1$ is the same as to give sections over the closed subsets $Z_+$ and $Z_-$ that agree on $Z_+ \cap Z_-$.}
	The section $\beta$ induces $Z_+ \to \mathscr A$ and $-\beta$ induces $Z_- \to \mathscr A$.  By pulling back our chosen trivialization of $\alpha^\ast W$, we obtain trivializations of $W_\beta \big|_{Z_+}$ and $W_{-\beta} \big|_{Z_-}$.  We observe that the composition~\eqref{eqn:45} is canonically zero, so these trivializations are opposite one another on $Z_+ \cap Z_-$:
	\setcounter{equation}{\value{theorem}}
	\begin{equation} \label{eqn:45}
		Z_+ \cap Z_- \subset Z_+ \times Z_- \xrightarrow{\beta \times -\beta} \mathscr A \times \mathscr A \xrightarrow{+} \mathscr A \to \tropGm
	\end{equation}
	That is, the trivialization of $W_{-\beta} \big|_{Z_-}$ is dual to a trivialization of the inverse torsor $W_\beta \big|_{Z_-}$ that agrees with the trivialization of $W_\beta \big|_{Z_+}$ on $Z_+ \cap Z_-$.  We therefore obtain compatible trivializations over $Z_+$ and $Z_-$ that glue to a trivialization over $Z_1$.  

	The map $Z_0 \to \tropGm$ is zero, so $W_\beta$ is canonically trivialized over $Z_0$.  Write $b$ for the canonical section of $W_\beta$ over $Z_1$ corresponding to this trivialization.  Let $a$ be the section of $W_\beta$ over $Z_1$ constructed in the last paragraph.  Let $i : Z_1 \to Z$ and $j : Z_0 \to Z$ be the inclusions.  We must show that the image of $a$ under $i^\ast W_\beta \to i^\ast j_\ast j^\ast W_\beta$ is $b$.  Since $Z_+$ and $Z_-$ cover $Z_1$, it is sufficient to vertify this after restriction to each of $Z_+$ and $Z_-$.  But the map $i^\ast W_\beta \big|_{Z_+} \to i^\ast (j_\ast j^\ast W_\beta) \big|_{Z_+}$ is pulled back from $\mathscr A$, where the desired identity holds by our original choice of the trivialization of $\alpha^\ast W$.

	The construction above involved no choices beyond the original choice of trivialization of $\alpha^\ast W$.  It is therefore compatible with base change and gives a trivialization of the underlying torsor of $W$ over $(\tropGm)_S$.  The structure of an extension is therefore encoded in a morphism $\varphi : (\tropGm \times \tropGm)_S \to (\tropGm)_S$ satisfying~\eqref{eqn:46} and~\eqref{eqn:47}:
	\begin{gather}
		\varphi(x+y,z) + \varphi(x,y) = \varphi(x,y+z) + \varphi(y+z) \label{eqn:46} \\
		\varphi(0,x) = \varphi(x,0) = 0 \label{eqn:47}
	\end{gather}
	But we know that $\varphi(x,y) = \lambda + nx + my$ for some $n,m : S \to \mathbf Z$ and $\lambda \in \Gamma(S, \overnorm M_S^{\rm gp})$, by Proposition~\ref{prop:tropGm-map}.  Substituting $x=0$ into~\eqref{eqn:47} gives $\lambda = 0$ and $m = 0$, and by symmetry $n = 0$ as well.
\end{proof}

\setcounter{theorem}{\value{equation}}
\begin{proposition} \label{prop:discrete-torsor}
	Let $T$ be a tropical torus over $S$ and let $Y$ be a lattice over $S$.  Then all morphisms $T \to Y$ factor through $S$ and all $Y$-torsors over $T$ descend uniquely to $S$.
\end{proposition}
\begin{proof}
	The assertions are local in $S$ so we assume without loss of generality that $T = (\tropGm^n)_S$.  We write $\pi : T \to S$ for the projection.  We will make use of a universal submersion $\varphi : (\mathbf P^1)_S^n \to (\tropGm^n)_S$, with $\mathbf P^1$ given its toric logarithmic structure.  If $Z$ is any logarithmic scheme over $(\tropGm^n)_S$, the (fine and saturated) base change of $\varphi$ is a $\Gm^n$-torsor over a logarithmic modification%
	\footnote{The map $Z \to \tropGm^n$ corresponds to a tuple of sections $(\alpha_1, \ldots, \alpha_n)$ of $\overnorm M_Z^{\rm gp}$.  The base change of $[ \mathbf P^1 / \Gm ]_S^n$ coincides with the iterated logarithmic blowup of $Z$ along the fractional ideals $(0, \alpha_1), \ldots, (0, \alpha_n)$.  By \cite[Lemma~3.6.13]{KU} it is a composition of logarithmic modifications, hence is a logarithmic modification.  The base change of $(\mathbf P^1)_S^n$ is a $\Gm^n$-torsor over this.}
	and in particular the reduced subschemes of its geometric fibers are $\Gm^n$-torsors over products of copies of $\mathbf P^1$.  In particular, it is a universal submersion (since it is the composition of a flat surjection and a proper surjection) and its fibers are geometrically connected.

	Suppose $f : T \to Y$.  We show that $f$ descends uniquely to $S$.  The composition $f \varphi : (\mathbf P^1)_S^n \to Y$ is constant on the fibers of $(\mathbf P^1)_S^n$ over $S$ since those fibers are connected and $Y$ is discrete.  Since $(\mathbf P^1)_S^n$ is proper over $S$, this implies that $f \varphi$ descends to a function $g : S \to Y$; the surjectivity of $(\mathbf P^1)_S^n \to S$ also implies that the descent of $f$ will be unique.  Since $\varphi$ is universally surjective, the identity $f \varphi = g \pi \varphi$ implies that $f = g \pi$.

	Now we prove that every $Y$-torsor on $T$ also descends uniquely to $S$.  It is equivalent to demonstrate that every $Y$-torsor on $T$ is \'etale-locally trivial in $S$ (since we have already demonstrated that morphisms $T \to Y$ descend uniquely to $S$).  If $Q$ is a $Y$-torsor on $T$ then $\varphi^\ast Q$ is a $Y$-torsor on $(\mathbf P^1)^n_S$.  By Lemma~\ref{lem:discrete-sect-tors}, $\varphi^\ast Q$ descends uniquely to $S$.  Replacing $S$ by an \'etale cover, we assume that $\varphi^\ast Q$ is trivial and we choose a section $q$.  We argue that $q$ descends to a section of $Q$ over $T$.

	Let $\alpha : Z \to T$ be an arbitrary morphism from a logarithmic scheme $Z$ and abuse $\varphi$ and $\alpha$ to stand also for the projections $W \to Z$ and $W \to (\mathbf P^1)_S^n$, respectively, where $W = Z \mathop\times_{(\tropGm^n)_S} (\mathbf P^1)^n_S$.  Then $\alpha^\ast q$ is a section of $\alpha^\ast \varphi^\ast Q = \varphi^\ast \alpha^\ast Q$.  We need to see that $\alpha^\ast q$ descends uniquely to $Z$.  This is an \'etale-local question in $Z$, so we assume that $\alpha^\ast Q$ is a trivial $Y$-torsor on $Z$.  Then we may view $q$ as a morphism $W \to Y$ by way of this isomorphism and this descends uniquely to $Z \to Y$ since $W \to Z$ is submersive (as in the proof of the first assertion of the proposition).  Since this descent is unique, it is compatible with variation of $Z$ and $\alpha : Z \to T$ and therefore defines a trivialization of $Q$ over $S$, as required.  
\end{proof}

\begin{proposition} \label{prop:discrete-ext}
	Let $T$ be a tropical torus over $S$ and let $Y$ be a lattice over $S$.  All extensions of $T$ by $Y$ are uniquely trivialized.
\end{proposition}
\begin{proof}
	The underlying torsor of an extension must be trivial by Proposition~\ref{prop:discrete-torsor} and the morphism $T \mathop\times_S T \to Y$ encoding the group structure must be constant, also by Proposition~\ref{prop:discrete-torsor}.

	The choices of trivialization of an extension form a torsor under the group of homomorphisms $T \to Y$.  But by Proposition~\ref{prop:discrete-torsor}, all maps $T \to Y$ factor through $S$, so all homomorphisms $T \to Y$ must be zero.
\end{proof}

\begin{proposition} \label{prop:log-hom-ext}
	Suppose that $T = \uHom(X,\logGm)$ and $T' = \uHom(X', \logGm)$ are logarithmic tori over a logarithmic scheme $S$.  Then $\uHom(T,T') = \uHom(X',X)$ and $\uExt^1(T,T') = 0$.
\end{proposition}
\begin{proof}
	It is shown in \cite[Proposition~2.5]{KKN2} that $\uHom(T,T') =
	\uHom(T^{\rm alg}, {T'}^{\rm alg})$, where $T^{\rm alg} = \uHom(X,\Gm)$
	and ${T'}^{\rm alg} = \uHom(X',\Gm)$ are the algebraic tori underlying
	$T$ and $T'$, respectively, and in \cite[Corollaire~1.4]{sga3-VIII}
	that $\uHom(T^{\rm alg}, {T'}^{\rm alg}) = \uHom(X',X)$.  This proves
	that $\uHom(T,T') = \uHom(X',X)$.  Alternatively, we can reduce to the
	case where $T$ and $T'$ are split and apply \cite[Corollary~11]{logGm}.

	For the second assertion, we work locally in $S$ and assume that $T$
	and $T'$ are split.  By the additivity of $\uExt^1$, we may therefore
	assume that $T = T' = (\logGm)_S$.  There is a commutative diagram with
	exact rows:
	\begin{equation*} \xymatrix{
			\uExt^1( (\tropGm)_S, (\Gm)_S ) \ar[r] \ar[d] & \uExt^1( (\tropGm)_S, (\logGm)_S ) \ar[r] \ar[d] & \uExt^1( (\tropGm)_S, (\tropGm)_S ) \ar[d] \\
			\uExt^1( (\logGm)_S, (\Gm)_S ) \ar[r] \ar[d] & \uExt^1( (\logGm)_S, (\logGm)_S ) \ar[r] \ar[d] & \uExt^1( (\logGm)_S, (\tropGm)_S ) \ar[d] \\
			\uExt^1( (\Gm)_S, (\Gm)_S ) \ar[r] & \uExt^1( (\Gm)_S, (\logGm)_S ) \ar[r] & \uExt^1( (\Gm)_S, (\tropGm)_S )
	} \end{equation*}
	We showed that $\uExt^1( (\tropGm)_S, (\tropGm)_S ) = 0$ in Proposition~\ref{prop:tropical-hom-ext}.  By \cite[Proposition~7.1.1]{sga3-XVII}, every extension of $(\Gm)_S$ by $(\Gm)_S$ is a torus, so $\uExt^1( (\Gm)_S, (\Gm)_S ) = 0$ as well.  Lemma~\ref{lem:discrete-sect-tors}, applied to the restriction of the underlying torsor to the small \'etale site of $(\Gm)_S$, shows that the underlying torsor of any extension of $\Gm$ by $\tropGm$ is trivial, so any extension of $\Gm$ by $\tropGm$ over $S$ is characterized by a map $\varphi : (\Gm^2)_S \to \tropGm$ that encodes a commutative group structure on $(\Gm \times \tropGm)_S$.  But all maps $(\Gm^2)_S \to \tropGm$ factor through $S$, so $\uExt^1( (\Gm)_S, (\tropGm)_S )$ vanishes also.  We obtain a surjective homomorphism:
	\begin{equation} \label{eqn:107}
		\uExt^1( (\tropGm)_S, (\Gm)_S ) \to \uExt^1( (\logGm)_S, (\logGm)_S )
	\end{equation}
	The images of $\uHom((\Gm)_S, (\Gm)_S) \simeq \mathbf Z$ and $\uHom((\tropGm)_S, (\tropGm)_S) \simeq \mathbf Z$ in $\uExt^1((\tropGm)_S, (\Gm)_S)$ coincide with the subgroup spanned by the extension $\logGm$ of $\tropGm$ by $\Gm$.  We will argue that this is all of $\uExt^1((\tropGm)_S, (\Gm)_S)$, which will imply that the surjection~\eqref{eqn:107} is zero, and therefore that $\uExt^1( (\logGm)_S, (\logGm)_S) = 0$. We will write $\mathcal O(n)$ for the image of $n \in \mathbf Z$ under the map $\mathbf Z = \uHom((\Gm)_S, (\Gm)_S) \to \uExt^1((\tropGm)_S, (\Gm)_S)$.  

	Suppose that $W$ is an extension of $\tropGm$ by $\Gm$.  We will begin
	by finding an integer $n$ such that $W(n) = W \otimes \mathcal O(n)$
	restricts to a trivial torsor on $[ \mathbf P^1 / \Gm ]_S$.  Then we
	will extend the trivialization of $W(n)$ to all logarithmic schemes
	over $S$.  Finally, we will show that there is a unique group structure
	on $W(n) \simeq \tropGm \times \Gm$ making it into an extension of
	$\tropGm$ by~$\Gm$.

	\textsc{Step 1.} The underlying torsor over $[\mathbf P^1 / \Gm ]$.  
	Give $\mathbf P^1$ its toric logarithmic structure and let 
	$\mathscr P = [ \mathbf P^1 / \Gm ]$.  
	Let $\alpha : \mathscr P_S \to \tropGm$ be the homomorphism
	corresponding to the section $(1,-1)$ of 
	$\Gamma(\mathscr P, \overnorm M^{\rm gp}_{\mathscr P}) \simeq \mathbf Z^2$.  
	There are two copies of $\mathrm B\Gm$ in $\mathscr P$, embedded as 
	$[\{ 0 \}/\Gm]$ and $[\{ \infty \}/\Gm]$.  We identify 
	$\Pic([\{0\}/\Gm]) = \Pic([\{\infty\}/\Gm]) = \mathbf Z$ so that the
	tangent spaces to $\mathbf P^1$ at $0$ and $\infty$ correspond to $+1$ 
	under the two identifications.%
	\footnote{This is different from the identification induced by the
	action of $\Gm$ on $\mathbf P^1$, but is consistent with the
	identification of the characteristic monoid of $\mathbf P^1$ with
	$\mathbf N$ at each point.}
	Restriction to these two loci and to the section of $\mathscr P_S$ corresponding
	to the open orbit of $\Gm$ on $\mathbf P^1$ gives an isomorphism 
	$\Pic(\mathscr P_S) \simeq \Pic(S) \times \mathbf Z^2$.  
	Let $\iota : \mathscr P \to \mathscr P$ be the map induced from multiplication
	by $-1$ on $\mathbf P^1$.  Then we have $\alpha \iota = [-1] \alpha$, so
	$\iota^\ast \alpha^\ast W = \alpha^\ast [-1]^\ast W$.  The action of $\iota^\ast$
	exchanges the two components of $\mathbf Z^2$ under the identification of
	$\Pic(\mathscr P) = \mathbf Z^2$ and the action of $[-1]^\ast$ reverses sign.
	Therefore $\alpha^\ast W$ lies in the antidiagonal of $\mathbf Z^2$.
	By construction, the pullback of $\mathcal O(1)$ to
	$\mathscr P$ has weights $+1$ at $0$ and $-1$ at $\infty$ so there is a unique
	$n$ such that $\alpha^\ast W(n)$ is pulled back from $S$.  In fact,
	$\alpha^\ast W(n)$ is then trivial, because its fiber over the open
	section of $\mathscr P$ over $S$ coincides with the fiber of $W$ over
	the origin of $\tropGm$, which is trivialized because $W$ is the torsor
	underlying an extension.

	\textsc{Step 2.} The underlying torsor over Artin fans.%
	\footnote{Artin fans are sometimes also called toric stacks or Olsson fans.}
	Replacing $W$ by $W(n)$, we will now assume that $\alpha^\ast W$ is trivial,
	and we fix one trivialization.  Our next task will be to propagate this
	trivialization to all stacks $\mathscr V_S = \mathscr V \times S$ where 
	$\mathscr V = [V/T]$ for a toric variety $V$ with dense torus $T$.

	Fix a map $\beta : \mathscr V \to \tropGm$.  Then $\mathscr V \mathop\times_{\tropGm} \mathscr P$ is a logarithmic modification $\tilde{\mathscr V}$ of $\mathscr V$.  We write $\tau : \tilde{\mathscr V} \to \mathscr V$ for the projection.  By \cite[p.\ 76, Proposition]{Fulton}, we have $\mathrm R \tau_\ast \mathcal O_{\tilde{\mathscr V}} = \mathcal O_{\mathscr V}$.  Since both $\mathscr V$ and $\tilde{\mathscr V}$ are flat over $\mathbf Z$ this also holds universally: $\mathrm R \tau_\ast \mathcal O_{\tilde{\mathscr V}_S} = \mathcal O_{\mathscr V_S}$.%
	\footnote{
		The claim $\mathrm R \tau_\ast \mathcal O_{\tilde{\mathscr V}_S} = \mathcal O_{\mathscr V_S}$ is \'etale-local in $S$.  We can therefore assume that $S \to \mathscr V$ factors through the toric variety $V$.  Let $\tilde V$ be the base change of $\tilde{\mathscr V}$ to $V$.  Then $\tilde V$ is a toric modification of $V$.  

		We wish to show that $\mathrm R \tau_\ast \mathcal O_{\tilde V_S} = \mathcal O_{V_S}$.  This claim is Zariski-local in $V$, so we may assume that $V$ is an affine toric variety: $V = \Spec A$.  Choose a finite cover of $\tilde V$ by open affines $U_i = \Spec B_i$ with intersections $U_{i_1} \cap \cdots \cap U_{i_n} = \Spec B_{i_1 \cdots i_n}$.  Since $H^p(\tilde V, \mathcal O_{\tilde V}) = 0$ for all $p > 0$, \v Cech cohomology gives an exact sequence:
		\begin{equation*}
			0 \to A \to \prod_i B_i \to \prod_{i < j} B_{ij} \to \cdots
		\end{equation*}
		This sequence has only finitely many nonzero terms, since the cover was selected to be finite, and all of the terms are flat over $\mathbf Z$.  Therefore the following sequence is also exact:
		\begin{equation*}
			0 \to A \otimes \mathcal O_S \to \prod B_i \otimes \mathcal O_S \to \prod_{i < j} B_{ij} \otimes \mathcal O_S \to \cdots
		\end{equation*}
		Since $V_S = \Spec_S (A \otimes \mathcal O_S)$, this shows that $\mathrm R \tau_\ast \mathcal O_{\tilde V_S} = \mathcal O_{V_S}$.
	}
	We obtain~\eqref{eqn:102} and therefore~\eqref{eqn:50}:
	\begin{gather} 
		\label{eqn:102}
		\tau_\ast \mathcal O_{\tilde{\mathscr V}_S}^\ast = \mathcal O_{\mathscr V_S}^\ast \\
		\label{eqn:50}
		\tau_\ast \tau^\ast \beta^\ast W = \beta^\ast W 
	\end{gather}
	If $\tilde\beta : \tilde{\mathscr V}_S \to \mathscr P_S$ denotes the base change
	of $\beta$ to $\tilde{\mathscr V}$, the bundle
	$\tau^\ast \beta^\ast W = \tilde\beta^\ast \alpha^\ast W$ inherits a trivialization
	by pullback from $\alpha^\ast W$.  Applying $\tau_\ast$ and~\eqref{eqn:50}, this 
	trivialization descends uniquely to a trivialization of $\beta^\ast W$ on $\mathscr V_S$.

	\textsc{Step 3.} The underlying torsor on logarithmic schemes over $S$.  
	We have trivialized $\beta^\ast W$ on every $\mathscr V_S = [V/T] \times S$ when $V$ is a toric variety and $T$ is its dense torus.  We will now extend this trivialization to all logarithmic schemes $U$ over $S$.  Provided we can extend this trivialization compatibly with respect to pullback, it is an \'etale-local problem to describe it.  We will therefore freely replace $U$ by an \'etale cover as necessary.

	Suppose that $U$ is a logarithmic scheme over $S$ and $\gamma \in \Gamma(U, \overnorm M_U^{\rm gp})$ gives a map $U \to \tropGm$.  Replacing $U$ by an \'etale cover if necessary, we can assume that $U$ has a global chart and factor $\gamma$ through some $\beta : \mathscr V \to \tropGm$ where $\mathscr V = [V/T]$ for an affine toric variety $V$ with dense torus $T$.  The trivialization of $\beta^\ast W$ on $\mathscr V_S$ constructed above pulls back to a trivialization of $\gamma^\ast W$.

	We must now prove that this trivialization is independent of the choice of factorization of $U \to \tropGm$ through $\beta : \mathscr V \to \tropGm$.  Suppose that $U \to \tropGm$ also factors through $\beta' : \mathscr V' \to \tropGm$, where $\mathscr V = [V/T]$ and $\mathscr V' = [V'/T']$ are quotients of toric varieties $V$ and $V'$ by their dense tori.  The identity of the two trivializations of $\gamma^\ast W$ is an \'etale-local question on $U$, so we may assume that both $V$ and $V'$ are \emph{affine} toric varieties, say $V = \Spec \mathbf Z[M]$ and $V' = \Spec \mathbf Z[M']$ for sharp, saturated monoids $M$ and $M'$.  It will suffice to show that $\mathscr V \mathop\times_{\tropGm} \mathscr V'$ is isomorphic to the quotient of a toric variety by its dense torus.  We may identify $M = \Gamma(\mathscr V, \overnorm M_{\mathscr V})$ and $M' = \Gamma(\mathscr V', \overnorm M_{\mathscr V'})$, so that $\beta \in M^{\rm gp}$ and $\beta' \in {M'}^{\rm gp}$.  Let $M''$ be the saturated image of $M \times M'$ in $M^{\rm gp} \times {M'}^{\rm gp} / \mathbf Z(\beta, -\beta')$.  Then $\mathscr V \mathop\times_{\tropGm} \mathscr V'$ is representable by $[V''/T'']$ where $V'' = \Spec \mathbf Z[M'']$ and $T''$ is its dense torus.  This completes the demonstration that $\gamma^\ast W$ is trivialized independent of choices (apart from the original trivialization of $\alpha^\ast W$ on $\mathscr P_S$).

	\textsc{Step 4.} The group structure.  
	We conclude that the underlying torsor of the extension $W$ is trivial.  Therefore the structure of an extension on $W$ is encoded by a morphism $\varphi : (\tropGm \times \tropGm)_S \to \Gm$ satistying $\varphi(x,0) = \varphi(0,x) = 1$.  But the induced map $(\alpha \times \alpha)^\ast \varphi : (\mathbf P^1 \times \mathbf P^1)_S \to \Gm$ factors through $S$ and is therefore constant with value $1$.  By the same process as in Step~2, we deduce that, for any Artin fan $\mathscr V$ and any map $\mathscr V \to \tropGm$, the map $(\mathscr V \times \mathscr V)_S \to \Gm$ induced by $\varphi$ must be constant with value~$1$.  Then, by the same process as in Step~3, we deduce that for any logarithmic scheme $U$ over $S$ and any $U \to \tropGm$, the induced map $U \mathop\times_S U \to \Gm$ is constant with value~$1$.  We conclude that $\varphi$ is constant with value~$1$ and therefore that the extension~$W$ is trivial.
\end{proof}

\bibliographystyle{halpha}
\bibliography{me,ega-sga,monodromy}

\begin{thebibliography}{SGA4-XVIII}
\expandafter\ifx\csname url\endcsname\relax
  \def\url#1{\texttt{#1}}\fi
\expandafter\ifx\csname doi\endcsname\relax
  \def\doi#1{\burlalt{doi:#1}{http://dx.doi.org/#1}}\fi
\expandafter\ifx\csname urlprefix\endcsname\relax\def\urlprefix{URL }\fi
\expandafter\ifx\csname href\endcsname\relax
  \def\href#1#2{#2}\fi
\expandafter\ifx\csname burlalt\endcsname\relax
  \def\burlalt#1#2{\href{#2}{#1}}\fi

\bibitem[AN09]{butterflies}
Ettore Aldrovandi and Behrang Noohi.
\newblock Butterflies. {I}. {M}orphisms of 2-group stacks.
\newblock {\em Adv. Math.}, 221(3):687--773, 2009.
\newblock \doi{10.1016/j.aim.2008.12.014}.

\bibitem[BB19]{BB19}
Cristiana Bertolin and Sylvain Brochard.
\newblock Morphisms of 1-motives defined by line bundles.
\newblock {\em Int. Math. Res. Not. IMRN}, (5):1568--1600, 2019.
\newblock \doi{10.1093/imrn/rny139}.

\bibitem[Ber11]{Ber11}
Cristiana Bertolin.
\newblock Extensions of {P}icard stacks and their homological interpretation.
\newblock {\em J. Algebra}, 331:28--45, 2011.
\newblock \doi{10.1016/j.jalgebra.2010.12.034}.

\bibitem[BL84]{BL2}
Siegfried Bosch and Werner L\"{u}tkebohmert.
\newblock Stable reduction and uniformization of abelian varieties. {II}.
\newblock {\em Invent. Math.}, 78(2):257--297, 1984.
\newblock \doi{10.1007/BF01388596}.

\bibitem[BL85]{BL1}
Siegfried Bosch and Werner L\"{u}tkebohmert.
\newblock Stable reduction and uniformization of abelian varieties. {I}.
\newblock {\em Math. Ann.}, 270(3):349--379, 1985.
\newblock \doi{10.1007/BF01473432}.

\bibitem[BLR90]{BLR}
Siegfried Bosch, Werner L\"{u}tkebohmert, and Michel Raynaud.
\newblock {\em N\'{e}ron models}, volume~21 of {\em Ergebnisse der Mathematik
  und ihrer Grenzgebiete (3) [Results in Mathematics and Related Areas (3)]}.
\newblock Springer-Verlag, Berlin, 1990.
\newblock \doi{10.1007/978-3-642-51438-8}.

\bibitem[BR15]{BR}
Matthew Baker and Joseph Rabinoff.
\newblock The skeleton of the {J}acobian, the {J}acobian of the skeleton, and
  lifting meromorphic functions from tropical to algebraic curves.
\newblock {\em Int. Math. Res. Not. IMRN}, (16):7436--7472, 2015.
\newblock \doi{10.1093/imrn/rnu168}.

\bibitem[Bro21]{Bro21}
Sylvain Brochard.
\newblock Duality for commutative group stacks.
\newblock {\em Int. Math. Res. Not. IMRN}, (3):2321--2388, 2021.
\newblock \doi{10.1093/imrn/rnz161}.

\bibitem[FC90]{FC}
Gerd Faltings and Ching-Li Chai.
\newblock {\em Degeneration of abelian varieties}, volume~22 of {\em Ergebnisse
  der Mathematik und ihrer Grenzgebiete (3) [Results in Mathematics and Related
  Areas (3)]}.
\newblock Springer-Verlag, Berlin, 1990.
\newblock \doi{10.1007/978-3-662-02632-8}.
\newblock With an appendix by David Mumford.

\bibitem[Ful93]{Fulton}
William Fulton.
\newblock {\em Introduction to toric varieties}, volume 131 of {\em Annals of
  Mathematics Studies}.
\newblock Princeton University Press, Princeton, NJ, 1993.
\newblock \doi{10.1515/9781400882526}.
\newblock The William H. Roever Lectures in Geometry.

\bibitem[Gil09]{Gillibert}
Jean Gillibert.
\newblock Prolongement de biextensions et accouplements en cohomologie log
  plate.
\newblock {\em Int. Math. Res. Not. IMRN}, (18):3417--3444, 2009.
\newblock \doi{10.1093/imrn/rnp059}.

\bibitem[HMOP20]{HMOP}
David Holmes, Samouil Molcho, Giulio Orecchia, and Thibault Poiret.
\newblock Models of {J}acobians of curves, 2020.
\newblock \doi{10.48550/ARXIV.2007.10792}.

\bibitem[Ill94]{Ill94}
Luc Illusie.
\newblock Autour du th\'{e}or\`eme de monodromie locale.
\newblock In {\em P\'{e}riodes $p$-adiques (Bures-sur-Yvette, 1988)}, number
  223 in Ast\'{e}risque, pages 9--57. 1994.

\bibitem[KKN08a]{KKN-pic}
Takeshi Kajiwara, Kazuya Kato, and Chikara Nakayama.
\newblock Analytic log {P}icard varieties.
\newblock {\em Nagoya Math. J.}, 191:149--180, 2008.
\newblock \doi{10.1017/S0027763000025940}.

\bibitem[KKN08b]{KKN2}
Takeshi Kajiwara, Kazuya Kato, and Chikara Nakayama.
\newblock Logarithmic abelian varieties.
\newblock {\em Nagoya Math. J.}, 189:63--138, 2008.
\newblock \doi{10.1017/S002776300000951X}.

\bibitem[KKN08c]{KKN1}
Takeshi Kajiwara, Kazuya Kato, and Chikara Nakayama.
\newblock Logarithmic abelian varieties. {I}. {C}omplex analytic theory.
\newblock {\em J. Math. Sci. Univ. Tokyo}, 15(1):69--193, 2008.

\bibitem[KKN15]{KKN4}
Takeshi Kajiwara, Kazuya Kato, and Chikara Nakayama.
\newblock Logarithmic abelian varieties, {P}art {IV}: {P}roper models.
\newblock {\em Nagoya Math. J.}, 219:9--63, 2015.
\newblock \doi{10.1215/00277630-3140577}.

\bibitem[KKN19]{KKN6}
Takeshi Kajiwara, Kazuya Kato, and Chikara Nakayama.
\newblock Logarithmic abelian varieties, {P}art {VI}: {L}ocal moduli and
  {GAGF}.
\newblock {\em Yokohama Math. J.}, 65:53--75, 2019.
\newblock \doi{10.18880/00013382}.

\bibitem[KN99]{KN}
Kazuya Kato and Chikara Nakayama.
\newblock Log {B}etti cohomology, log \'{e}tale cohomology, and log de {R}ham
  cohomology of log schemes over {${\bf C}$}.
\newblock {\em Kodai Math. J.}, 22(2):161--186, 1999.
\newblock \doi{10.2996/kmj/1138044041}.

\bibitem[KS94]{KS}
Masaki Kashiwara and Pierre Schapira.
\newblock {\em Sheaves on manifolds}, volume 292 of {\em Grundlehren der
  mathematischen Wissenschaften [Fundamental Principles of Mathematical
  Sciences]}.
\newblock Springer-Verlag, Berlin, 1994.
\newblock With a chapter in French by Christian Houzel, Corrected reprint of
  the 1990 original.

\bibitem[KU09]{KU}
Kazuya Kato and Sampei Usui.
\newblock {\em Classifying spaces of degenerating polarized {H}odge
  structures}, volume 169 of {\em Annals of Mathematics Studies}.
\newblock Princeton University Press, Princeton, NJ, 2009.
\newblock \doi{10.1515/9781400837113}.

\bibitem[L{\"u}t16]{L}
Werner L{\"u}tkebohmert.
\newblock {\em Rigid geometry of curves and their {J}acobians}, volume~61 of
  {\em Ergebnisse der Mathematik und ihrer Grenzgebiete. 3. Folge. A Series of
  Modern Surveys in Mathematics [Results in Mathematics and Related Areas. 3rd
  Series. A Series of Modern Surveys in Mathematics]}.
\newblock Springer, Cham, 2016.
\newblock \doi{10.1007/978-3-319-27371-6}.

\bibitem[Mum72]{Mumford-72b}
David Mumford.
\newblock An analytic construction of degenerating abelian varieties over
  complete rings.
\newblock {\em Compositio Math.}, 24:239--272, 1972.

\bibitem[MW18]{logpic}
Samouil Molcho and Jonathan Wise.
\newblock The logarithmic picard group and its tropicalization, 2018.
\newblock \doi{10.48550/ARXIV.1807.11364}.

\bibitem[Nak17]{Nakayama}
Chikara Nakayama.
\newblock Logarithmic étale cohomology, ii.
\newblock {\em Advances in Mathematics}, 314:663--725, 2017.
\newblock \doi{https://doi.org/10.1016/j.aim.2017.05.006}.

\bibitem[Ols03]{Olsson03}
Martin~C. Olsson.
\newblock Logarithmic geometry and algebraic stacks.
\newblock {\em Ann. Sci. \'{E}cole Norm. Sup. (4)}, 36(5):747--791, 2003.
\newblock \doi{10.1016/j.ansens.2002.11.001}.

\bibitem[Ray71]{Raynaud-ICM}
Michel Raynaud.
\newblock Vari\'{e}t\'{e}s ab\'{e}liennes et g\'{e}om\'{e}trie rigide.
\newblock In {\em Actes du {C}ongr\`es {I}nternational des {M}ath\'{e}maticiens
  ({N}ice, 1970), {T}ome 1}, pages 473--477. 1971.

\bibitem[Ray94]{Raynaud}
Michel Raynaud.
\newblock 1-motifs et monodromie g\'{e}om\'{e}trique.
\newblock In {\em P\'{e}riodes $p$-adiques (Bures-sur-Yvette, 1988)}, number
  223 in Ast\'{e}risque, pages 295--319. 1994.

\bibitem[RW20]{logGm}
Dhruv Ranganathan and Jonathan Wise.
\newblock Rational curves in the logarithmic multiplicative group.
\newblock {\em Proc. Amer. Math. Soc.}, 148(1):103--110, 2020.
\newblock \doi{10.1090/proc/14749}.

\bibitem[SGA3-XVII]{sga3-XVII}
Michel Raynaud.
\newblock Groupes alg\'{e}briques unipotents. extensions entre groupes
  unipotents et groupes de type multiplicatif.
\newblock In {\em Sch\'{e}mas en groupes. {II}: {G}roupes de type
  multiplicatif, et structure des sch\'{e}mas en groupes g\'{e}n\'{e}raux},
  Lecture Notes in Mathematics, Vol. 152, pages 531--631. Springer-Verlag,
  Berlin-New York, 1970.
\newblock S\'{e}minaire de G\'{e}om\'{e}trie Alg\'{e}brique du Bois Marie
  1962/64 (SGA 3), Dirig\'{e} par M. Demazure et A. Grothendieck.

\bibitem[SGA3-VIII]{sga3-VIII}
Alexandre Grothendieck.
\newblock Groupes diagonalisables.
\newblock In {\em Sch\'{e}mas en groupes. {II}: {G}roupes de type
  multiplicatif, et structure des sch\'{e}mas en groupes g\'{e}n\'{e}raux},
  Lecture Notes in Mathematics, Vol. 152, pages 1--36. Springer-Verlag,
  Berlin-New York, 1970.
\newblock S\'{e}minaire de G\'{e}om\'{e}trie Alg\'{e}brique du Bois Marie
  1962/64 (SGA 3), Dirig\'{e} par M. Demazure et A. Grothendieck.

\bibitem[SGA4-XVIII]{sga4-XVIII}
Pierre Deligne.
\newblock La formule de dualit\'e globale.
\newblock In {\em Th\'{e}orie des topos et cohomologie \'{e}tale des
  sch\'{e}mas. {T}ome 3}, Lecture Notes in Mathematics, Vol. 305, pages
  481--587. Springer-Verlag, Berlin-New York, 1973.
\newblock S\'{e}minaire de G\'{e}om\'{e}trie Alg\'{e}brique du Bois-Marie
  1963--1964 (SGA 4), Dirig\'{e} par M. Artin, A. Grothendieck et J. L.
  Verdier. Avec la collaboration de P. Deligne et B. Saint-Donat.

\bibitem[SGA4-IV]{sga4-IV}
A.~Grothendieck and J.~L. Verdier.
\newblock {\em Topos}.
\newblock Lecture Notes in Mathematics, Vol. 269. Springer-Verlag, Berlin-New
  York, 1972.
\newblock S\'{e}minaire de G\'{e}om\'{e}trie Alg\'{e}brique du Bois-Marie
  1963--1964 (SGA 4), Dirig\'{e} par M. Artin, A. Grothendieck, et J. L.
  Verdier. Avec la collaboration de N. Bourbaki, P. Deligne et B. Saint-Donat.

\bibitem[SGA7-VII]{sga7-VII}
Alexandre Grothendieck.
\newblock Biextensions de faisceaux de groupes.
\newblock In {\em Groupes de monodromie en g\'{e}om\'{e}trie alg\'{e}brique.
  {I}}, Lecture Notes in Mathematics, Vol. 288, pages 133--217.
  Springer-Verlag, Berlin-New York, 1972.
\newblock S\'{e}minaire de G\'{e}om\'{e}trie Alg\'{e}brique du Bois-Marie
  1967--1969 (SGA 7 I), Dirig\'{e} par A. Grothendieck. Avec la collaboration
  de M. Raynaud et D. S. Rim.

\bibitem[SGA7-VIII]{sga7-VIII}
Alexandre Grothendieck.
\newblock Compl\'ements sur les biextensions. propri\'et\'es g\'en\'erales des
  biextensions des sch\'emas en groupes.
\newblock In {\em Groupes de monodromie en g\'{e}om\'{e}trie alg\'{e}brique.
  {I}}, Lecture Notes in Mathematics, Vol. 288, pages 218--312.
  Springer-Verlag, Berlin-New York, 1972.
\newblock S\'{e}minaire de G\'{e}om\'{e}trie Alg\'{e}brique du Bois-Marie
  1967--1969 (SGA 7 I), Dirig\'{e} par A. Grothendieck. Avec la collaboration
  de M. Raynaud et D. S. Rim.

\bibitem[SGA7-IX]{sga7-IX}
Alexandre Grothendieck.
\newblock Mod\`{e}les de n\'{e}ron et monodromie.
\newblock In {\em Groupes de monodromie en g\'{e}om\'{e}trie alg\'{e}brique.
  {I}}, Lecture Notes in Mathematics, Vol. 288, pages 313--523.
  Springer-Verlag, Berlin-New York, 1972.
\newblock S\'{e}minaire de G\'{e}om\'{e}trie Alg\'{e}brique du Bois-Marie
  1967--1969 (SGA 7 I), Dirig\'{e} par A. Grothendieck. Avec la collaboration
  de M. Raynaud et D. S. Rim.

\bibitem[SZ85]{SZ85}
Joseph Steenbrink and Steven Zucker.
\newblock Variation of mixed {H}odge structure. {I}.
\newblock {\em Invent. Math.}, 80(3):489--542, 1985.
\newblock \doi{10.1007/BF01388729}.

\bibitem[Wis16]{minimal}
Jonathan Wise.
\newblock Moduli of morphisms of logarithmic schemes.
\newblock {\em Algebra Number Theory}, 10(4):695--735, 2016.
\newblock \doi{10.2140/ant.2016.10.695}.

\bibitem[Wis21]{2ab}
Jonathan Wise.
\newblock The 2-category of 2-term complexes, 2021.
\newblock \doi{10.48550/ARXIV.2107.13667}.

\bibitem[Zha21]{Zhao}
Heer Zhao.
\newblock Extending tamely ramified strict 1-motives into k\'{e}t log
  1-motives.
\newblock {\em Forum Math. Sigma}, 9:Paper No. e20, 34, 2021.
\newblock \doi{10.1017/fms.2021.5}.

\end{thebibliography}

\end{document}